\documentclass{scrartcl}
\usepackage{amsmath,amssymb}
\usepackage{dsfont,amsthm}

\newcommand{\N}{{\mathds N}}
\newcommand{\R}{{\mathds R}}

\newcommand{\Var}{{\rm Var}}
\newcommand{\Cov}{{\rm Cov}}

\newtheorem{lemma}{Lemma}[section]
\newtheorem{proposition}[lemma]{Proposition}
\newtheorem{example}[lemma]{Example}
\newtheorem{remark}[lemma]{Remark}
\newtheorem{definition}[lemma]{Definition}
\newtheorem{theorem}{Theorem}[section]

\begin{document}

\pagestyle{myheadings}
\markright{Bootstrap for Near Epoch Dependent Processes}

\title{Bootstrap for the Sample Mean and for $U$-Statistics of Mixing and Near Epoch Dependent Processes}

\author{Olimjon Sh. Sharipov\thanks{Olimjon Sharipov was supported by DFG (German Research Foundation).} \thanks{Institute of Mathematics and Information Technologies, Uzbek Academy of Sciences, 29 Dormon Yoli Str., Tashkent, 100125, Uzbekistan} \and Martin Wendler \thanks{Martin Wendler was supported by the Studienstiftung des deutschen Volkes.} \thanks{Fakult\"{a}t f\"{u}r Mathematik, Ruhr-Universit\"{a}t Bochum, 44780 Bochum, Germany,  Email address:  Martin.Wendler@rub.de}}
\date{\today}

\maketitle

\begin{abstract}
The validity of various bootstrapping methods has been proved for the sample mean of strongly mixing data. But in many applications, there appear nonlinear statistics of processes that are not strongly mixing. We investigate the nonoverlapping block bootstrap sequences which are near epoch dependent on strong mixing or absolutely regular processes. This includes ARMA and GARCH-processes as well as data from chaotic dynamical systems. We establish the strong consistency of the bootstrap distribution estimator not only for the sample mean, but also for $U$-statistics, which include examples as Gini's mean difference or the $\chi^2$-test statistic.
\end{abstract}

\noindent {\itshape keywords:} strongly mixing sequences, near epoch dependence, $U$-statistics, block bootstrap


\section{Introduction}

\subsection{Dependent Random Variables}

In many statistical applications the data does not come from an independent stochastic process. A standard assumption of weak dependence is given by the strong mixing condition:
\begin{definition} Let $\left(X_n\right)_{n\in\N}$ be a stationary process. Then the strong mixing coefficient is given by
\begin{equation*}
\alpha (k) = \sup \left\{\left| P (A\cap B) - P (A) P (B) \right| : A \in \mathcal{F}^n_1, B \in \mathcal{F}^\infty_{n+k}, n \in \N \right\},
\end{equation*}
where $\mathcal{F}^l_a$ is the $\sigma$-field generated by r.v.'s $X_a, \ldots, X_l.$, and $\left(X_n\right)_{n\in\N}$ is called strongly mixing, if $\alpha (k)\rightarrow0$ as $k\rightarrow\infty.$
\end{definition}
For further information on strong mixing and a detailed description of other mixing conditions see Doukhan \cite{douk} and Bradley \cite{bra2}. However, this class of weak dependent processes excludes examples like linear processes with innovations that do not have a density or data from dynamical systems.

\begin{example}\label{exa1} Let $\left(Z_n\right)_{n\in\N}$ be independent r.v.'s with $P\left[Z_n=1\right]=P\left[Z_n=0\right]=\frac{1}{2}$ and
\begin{equation*}
 X_n=\sum_{k=n}^{\infty}\frac{1}{2^{k-n+1}}Z_k.
\end{equation*}
Then $X_{n=1}=2X_n$ [mod 1] and $\left(X_n\right)_{n\in\N}$ is not strong mixing, as
\begin{multline*}
\left|P\left[X_1\in\bigcup_{i=1}^{2^{(k-1)}}\left[(2i-2)2^{-k},(2i-1)2^{-k}\right],\ X_k\in\left[0,\frac{1}{2}\right]\right]\right.\\
\left.-P\left[X_1\in\bigcup_{i=1}^{2^{(k-1)}}\left[(2i-2)2^{-k},(2i-1)2^{-k}\right]\right]P\left[X_k\in\left[0,\frac{1}{2}\right]\right]\right|
=\frac{1}{2}-\frac{1}{2}\cdot\frac{1}{2}=\frac{1}{4}.
\end{multline*}
\end{example}
We will consider sequences which are near epoch dependent on strongly mixing or absolutely regular processes, as this class covers the example above and other interesting examples and data from other dynamical systems, which are deterministic except for the initial value. Near epoch dependent functionals of mixing processes have been studied for a long time, see for example Ibragimov \cite{ibra} or Billingsley \cite{bill}.

\begin{definition} Let $\left(X_n\right)_{n\in\N}$ be a stationary process.
\begin{enumerate}
 \item The absolute regularity coefficient is given by
\begin{equation*}
\beta (k) = \sup_{n\in\N}E \sup \{ \left| P (A / \mathcal{F}_{-\infty}^n) - P (A) \right| : A \in \mathcal {F}^\infty_{n + k}\},
\end{equation*}
and $\left(X_n\right)_{n\in\N}$ is called absolutely regular, if $\beta(k)\rightarrow0$ as $k\rightarrow\infty.$
\item We say that $\left(X_n\right)_{n\in\mathds{N}}$ is near epoch dependent on a process $(Z_n)_{n\in\mathds{Z}}$ with approximation constants $(a_l)_{l\in\mathds{N}}$, if
\begin{equation*}\
E \left| X_1 - E (X_1 | \mathcal {F}^l_{- l}) \right| \leq a_l \qquad l = 0, 1,2 \ldots
\end{equation*}
where $\lim_{l \rightarrow\infty} a_l = 0$ and $ \mathcal {F}_{-l}^l$ is the $\sigma$-field generated by $Z_{-l}, \ldots, Z_l.$
\end{enumerate}
\end{definition}

We consider $L^1$ near epoch dependence, which is a weaker assumption than the more commonly used $L^2$ near epoch dependence.

\begin{example}Expanding Dynamical Systems: Let $T:\left[0,1\right]\rightarrow\left[0,1\right]$ be a piecewise smooth and expanding map such that $\inf_{x\in\left[0,1\right]}\left|T'\left(x\right)\right|>1$. Then there is a stationary process $\left(X_n\right)_{n\in\N}$ such that $X_{n+1}=T\left(X_n\right)$ which can be represented as a functional of an absolutely regular process (see Hofbauer, Keller \cite{hofb}).
\end{example}

\begin{example}GARCH: Let $(Z_n)_{n\in\N}$ be a sequence of independent standard normal random variables and $X_n=\sigma_nZ_n$, where $(\sigma_n)_{n\in\N}$ is a random sequence with $\sigma_n^2=\alpha_0+\alpha_1X_{n-1}^2+\alpha_2\sigma_{n-1}^2$. Such GARCH(1,1) processes are used for modeling volatility clustering in financial data and are near epoch dependent with an exponential decay of the approximation constants, see Hansen \cite{hans}. 
\end{example}

\begin{example}Volterra series: Many causal processes can be represented by a Taylor series. Assume that $(Z_n)_{n\in\N}$ is a sequence of independent identically distributed and centered random variables with finite variance. We define
\begin{equation*}
X_n=\sum_{l=1}^\infty\sum_{u_1,\ldots,u_l=0}^\infty g_l(u_1,\ldots,u_l)Z_{n-u_1}\ldots Z_n{n-u_l}.
\end{equation*}
Let $g_l(u_1,\ldots,u_l)=0$ if $u_i=u_j$ for some $i,j$, then $X_n$ is near epoch dependent with approximation constants
\begin{equation*}
a_k\leq C\sqrt{\sum_{l=1}^\infty\sum_{u_1,\ldots,u_l=k}^\infty g^2(u_1,\ldots,u_l)}.
\end{equation*}
Volterra series can be understood as Taylor expansions for causual time series, so they include a broad class of models. For more details, see Rugh \cite{rugh}.
\end{example}

\subsection{$U$-Statistics}

$U$-statistics play an important role in nonparametric statistics because many estimators and test statistics can be written at least asymptotically as $U$-statistics. Well-known examples include the sample variance, Gini's mean difference, and the $\chi^2$ goodness of fit test statistic. A more recent example is the Grassberger-Procaccia dimension estimator. $U$-statistics can be described as generalized means, i.e. means of the values of a kernel function $h\left(X_{i_1},\ldots,X_{i_k}\right)$. For simplicity of notation, we concentrate on the case of bivariate $U$-statistics:

\begin{definition}A $U$-statistic with a symmetric and measurable kernel $h:\mathds{R}^2\rightarrow\mathds{R}$ is defined as
\begin{equation*}
 U_{n}\left(h\right)=\frac{2}{n\left(n-1\right)}\sum_{1\le i<j\leq n}h\left(X_{i},X_{j}\right).
\end{equation*}
\end {definition}

The key tool in the analysis of $U$-statistics is the Hoeffding-decomposition \cite{hoef} of $U_n\left(h\right)$ into a so-called linear part and a degenerate part
\begin{equation*}
U_{n}\left(h\right)=\theta+\frac{2}{n}\sum_{i=1}^{n}h_{1}\left(X_{i}\right)+U_n\left(h_2\right)
\end{equation*}
with
\begin{align*}
\theta&:=Eh\left(X,Y\right),\\
h_1(x)&:=Eh(x,Y)-\theta, \\
h_2(x,y)&:=h(x,y) - h_1(x) -h_1(y) -\theta.
\end{align*}
for $X$, $Y$ independent and with the same distribution as $X_1$. We will need some more technical conditions on the kernel:

\begin{definition}\label{def1} Let $\left(X_n\right)_{n\in\mathds{N}}$ be a stationary process.
 \begin{enumerate}
 \item A kernel has uniform $r$-moments, if there is a $M>0$ such that for all $k\in\mathds{N}_{0}$
\begin{equation*}
E\left|h\left(X_{1},X_{k}\right)\right|^{r}\leq M\quad\text{and}\quad E\left|h\left(X,Y\right)\right|^{r}\leq M
\end{equation*}
for $X$, $Y$ independent and with the same distribution as $X_1$.
\item A kernel $h$ is called $P$-Lipschitz-continuous with constant $L>0$, if
\begin{equation*}
 E\left[\left|h\left(X,Y\right)-h\left(X',Y\right)\right|\mathds{1}_{\left\{\left|X-X'\right|\leq\epsilon\right\}}\right]\leq L\epsilon
\end{equation*}
for every $\epsilon>0$, every pair $(X,Y)$ with the same common distribution as $(X_1,X_k)$ for some $k\in\mathds{N}$ or independent with the same distribution as $X_1$ and $(X',Y)$ also with one of these common distributions. With $\mathds{1}_A$, we denote the indicator function of a set $A$.
 \end{enumerate}
\end{definition}

The Lipschitz-continuity condition is rather mild, some kernel with junps satisfy it, see Dehling, Wendler \cite{dehl}. Furthermore, we have the following:
\begin{lemma}
 \begin{enumerate}
\item Let $h$ be a polynomial kernel of degree $d$, that is
\begin{equation*}
h\left(x,y\right)=\sum_{i=0}^d\sum_{j=0}^{d-i}c_{ij}\left(x^iy^j+x^jy^i\right).
\end{equation*}
If $E\left|X_1\right|^{d-1}<\infty$, then $h$ is $P$-Lipschitz-continuous.
\item Let $h$ be a $P$-Lipschitz-continuous kernel and $f:\R\rightarrow\R$ a Lipschitz-continuous function. Then $g\circ h$ is $P$-Lipschitz-continuous.
 \end{enumerate}
\end{lemma}

\begin{proof}
 \begin{enumerate}
 \item We can concentrate on an expression of the form $g\left(x,y\right)=x^iy^j$, $i+j\leq d$:
\begin{align*}
&E\left[\left|g\left(X,Y\right)-g\left(X',Y\right)\right|\mathds{1}_{\left\{\left|X-X'\right|\leq\epsilon\right\}}\right]\\
=&E\left[\left|\left(X-X'\right)\left(X^{i-1}+X^{i-2}X'+\ldots+XX'^{i-2}+X'^{i-1}\right)Y^j\right|\mathds{1}_{\left\{\left|X-X'\right|\leq\epsilon\right\}}\right]\\
\leq& \epsilon E\left[\left|\left(X^{i-1}+X^{i-2}X'+\ldots+XX'^{i-2}+X'^{i-1}\right)Y^j\right|\right]\leq\epsilon iE\left|X_1\right|^{i+j}
\end{align*}
\item This is obvious.
 \end{enumerate}
\end{proof}

Whereas the limit theory for partial sums of weakly dependent processes is very well developed, much less attention has been paid to nonlinear statistics like $U$-statistics. The summands of $U_n\left(h_2\right)$ can be correlated, if the random variables $\left(X_n\right)_{n\in\mathds{N}}$ are dependent, so one has to establish generalized covariance inequalities to derive moment bounds for $U_n\left(h_2\right)$. Yoshihara \cite{yosh} considered absolutely regular processes, Denker and Keller \cite{denk} functionals of absolutely regular processes, and Dehling and Wendler \cite{dehl} strongly mixing sequences.

\subsection{Block Bootstrap}

In many statistical applications, for example in the determination of confidence bands, one faces the task to compute the distribution of a statistic $T_n=T_n\left(X_1,\ldots,X_n\right)$. This is a challenging problem if the random variables are dependent and the function $T_n$ is nonlinear. Therefore, block bootstrapping method are commonly used for nonparametric inference. There are different ways to resample blocks, for example the circular block bootstrap or the moving block bootstrap (for a detailed description of the different bootstrapping methods see Lahiri \cite{lahi}). For the circular block bootstrap, Shao and Yu \cite{shao} have shown that under strong mixing the distribution of the block bootstrap version $\bar{X}^\star_n$ of the sample mean converges almost surely to the same distribution as the sample mean $\bar{X}_n$. Peligrad \cite{peli} has proved asymptotic normality of $\bar{X}_n^\star$ under another set of conditions, which does not necessarily imply the central limit theorem for $\bar{X}_n$. Radulovic \cite{radu} has established weak consistency under very weak conditions.

Gon\c{c}alves and White \cite{gonc} have proved the weak consistency under near epoch dependence. As far as we know, strong consistency has not been proved under such conditions, neither for the sample mean nor for $U$-statistics.

We consider the nonoverlapping block bootstrap, proposed by Carlstein \cite{carl}, for the sample mean and for $U$-statistics. Let $\left(X_n\right)_{n\in\N}$ be a sequence of r.v.'s. Let $p\in\mathds{N}$ be the block length such that $ p = p (n) = o (n)$, $p \rightarrow \infty$ as $ n \rightarrow \infty$. We introduce the following blocks of indices and r. v.'s:
\begin{align*}
{\rm I}_i&= \left(X_{(i-1)p+1}, \ldots, X_{ip}\right),\displaybreak[0]\\
{\rm B}_i&=\left\{(i-1)p+1, \ldots, ip\right\},\quad i = 1, \ldots, k
\end{align*}
where $k = k (n) = \left[\frac{n}{p}\right]$ is the number of blocks. We consider a new sample $ X_1^*, \ldots, X_{kp}^*$, which is constructed by choosing randomly and independently blocks $k$ times with
\[
{\rm P}\left((X_1^*, \ldots, X_p^*)={\rm I}_i\right) =\frac{1}{{k}} \quad i=1,2, \ldots,k.
\]
As a bootstrap version of the sample mean we consider: 
\[
\bar{X}^*_{n,kp} = \frac{1}{kp}\sum^{kp}_{i=1} X^*_i.
\]
With $P^\star$, $E^\star$, $\Var^\star$ we denote the probability, expectation and variance conditionally on $\left(X_n\right)_{n\in\N}$. Note that
\begin{equation*}
{\rm E}^* \bar{X}^*_{n,kp} = \frac{1}{kp} \sum^{kp}_{i=1} X_{i} =: \bar{X}_{n,kp}.
\end{equation*}

The first aim of this paper is to prove the weak and the strong consistency of the nonoverlapping block bootstrap for sequences that are near epoch dependent on strongly mixing processes (Theorems \ref{theo6} and \ref{theo7}), as this class of weak dependent processes covers examples that do not satisfy the strong mixing conditions. 

Our second aim is to prove the weak and the strong consistency of the nonoverlapping block bootstrap for $U$-statistics. Although the estimation of the distribution for $U$-statistics is even more complicated than for the sample mean, there is only very little literature on the bootstrap for $U$-statistics. Bickel and Freedman \cite{bick} proved the validity of the bootstrap for nondegenerate $U$-statistcs of i.i.d. data, Arcones, Gin\'e \cite{arco}, Dehling, Mikosch \cite{deh3}, and Leucht, Neumann \cite{leuc} for degenerate $U$-statistics of i.i.d. data. Dehling and Wendler \cite{dehl} have shown that the weak constistency of the circular block bootstrap for nondegenerate $U$-statistics of strongly mixing or absolutely regular sequences.

\section{Main Results}

In this section, and in what follows, we denote by $\bar{X}_n$ the sample mean of the observations $X_1, \ldots,X_n,$ by $N(0, \sigma^2)$ a Gaussian r.v. with mean zero and variance $\sigma^2$ and by $C$ a constant which may depend on several parameters and might have different values even in one chain of inequalities. First we will give theorems for the nonoverlapping block bootstrap and general stationary sequences which are analogues to the results of Peligrad \cite{peli}, and Shao and Yu \cite{shao} for the circular block bootstrap.

\subsection{Bootstrap for the sample mean}

We formulate theorems for the nonoverlapping block bootstrap. We will first show the weak constistency under strong mixing and near epoch dependence on a strongly mixing process. The first part of the following theorem (strong mixing) can be found in the book of Lahiri \cite{lahi}.

\begin{theorem}\label{theo5}Let $\left(X_n\right)_{n\in\mathds{N}}$ be a stationary sequence r.v.'s with $E X_1 = \mu$, for some $\delta > 0$: $E \left| X_1\right|^{2 + \delta} < \infty$. Assume that $1/p(n)+p(n)/n\rightarrow0$ as $n\rightarrow\infty$ and that one of the two following conditions holds
\begin{enumerate}
 \item $\left(X_n\right)_{n\in\mathds{N}}$ is strongly mixing with $\sum_{k=1}^\infty\alpha^{\frac{\delta}{2+\delta}} (k) < \infty$. 
\item $\left(X_n\right)_{n\in\mathds{N}}$ is near epoch dependent (with approximation constants $\left( a_l\right)_{l\in\N}$) on a strongly mixing process process $ \left(Z_n\right)_{n\in\mathds{Z}}$ and $\sum^\infty_{k = 1}( a_k^{\frac{\delta}{1 + \delta}} + \alpha^{\frac{\delta}{2 + \delta}} (k)) < \infty$.
\end{enumerate}
Then $\sigma^2 = EX^2_1 + 2 \sum^\infty_{i =2} \Cov (X_1, X_i) <\infty$ and in the case $\sigma^2 > 0$
\begin{align}
\Var^* (\sqrt{kp} \bar{X}^*_{n,kp}) \rightarrow \sigma^2 \qquad &\mbox { in probability}\\
 \sup_{x\in\R} \left| P^*\left(\sqrt{kp} (\bar{X}_{n,kp}^* - \bar{X}_{n,kp}) \leq x \right) -P \left( \sqrt{n} (\bar{X}_n - \mu) \leq x\right)\right| \rightarrow 0 \qquad &\mbox{ in probability}
\end{align}

\end{theorem}

\begin{remark} This Theorem is also valid for the moving or the circular block bootstrap, we skip the introduction of these two bootstrap methods. This Theorem weakens the moment assumption of Gon\c{c}alves and White \cite{gonc} (they considered not only sequences, but also triangular arrays.)
\end{remark}

The first part of the following theorem (strong mixing) is analogues to the results of Peligrad \cite{peli} and Shao, Yu \cite{shao} for the circular block bootstrap, showing strong consistency.
\begin{theorem}\label{theo6}Let $\left(X_n\right)_{n\in\mathds{N}}$ be a stationary sequence r.v.'s with $E X_1 = \mu$, $p(n)\rightarrow\infty$,
\begin{equation}\label{line12}
p (n) \leq C n^\epsilon \mbox{ for some } 0 < \epsilon < 1 \mbox { and }
\end{equation}
\begin{equation}\label{line13}
p (n) = p(2^l) \mbox { for } 2^l < n \leq 2^{l + 1},\qquad l = 1, 2, \ldots
\end{equation}
and assume that one of the two following conditions holds:
\begin{enumerate}
 \item Let be $(E\left|X_1\right|^{2 + \delta})^{\frac{1}{2 + \delta}} < \infty$ for some $ 0 < \delta \leq \infty$. Assume that $\left(X_n\right)_{n\in\mathds{N}}$ is strongly mixing with $\alpha(k)=O(k^{-\alpha})$ for an $\alpha>\frac{2+\delta}{\delta}$.
\item Let be $E \left| X_1\right|^{4 + \delta} < \infty$ for some $\delta > 0$. Assume that $\left(X_n\right)_{n\in\mathds{N}}$ is near epoch dependent (with approximation constants $\left( a_l\right)_{l\in\N}$) on a stationary absolutely regular process $ \left(Z_n\right)_{n\in\mathds{Z}}$ and $\sum^\infty_{k = 0} k^2 ( a_k^{\frac{\delta}{3 + \delta}} + ( \beta (k))^{\frac{\delta}{4 + \delta}}) < \infty$.
\end{enumerate}
Then $\sigma^2 = EX^2_1 + 2 \sum^\infty_{i =2} \Cov (X_1, X_i) <\infty$ and in the case $\sigma^2 > 0$
\begin{align}
\Var^* (\sqrt{kp} \bar{X}^*_{n,kp}) \rightarrow \sigma^2 \qquad &\mbox { a.s.}\\
\sup_{x\in\R} \left| P^*\left(\sqrt{kp} (\bar{X}_{n,kp}^* - \bar{X}_{n,kp}) \leq x \right) -P \left( \sqrt{n} (\bar{X}_n - \mu) \leq x\right)\right| \rightarrow 0 \qquad &\mbox{ a.s. }
\end{align}
\end{theorem}

\begin{remark}Line (\ref{line13}) is a technical condition on the block lenght that is needed for the chaining techniques we will use. The conditions in the second part of this Theorem are the same as for the Central Limit Theorem in Borovkova et al. \cite{boro}.
\end{remark}

\begin{theorem}\label{theo7}Let $\left(X_n\right)_{n\in\mathds{N}}$ be a stationary sequence of almost surely bounded r.v.'s. Assume that $p(n)\rightarrow\infty$,
\begin{equation*}
p (n) = p(2^l) \mbox { for } 2^l < n \leq 2^{l + 1},\qquad l = 1, 2, \ldots,
\end{equation*}
$p^2/n\rightarrow0$, and one the two following conditions holds
\begin{enumerate}
 \item $\left(X_n\right)_{n\in\mathds{N}}$ is strongly mixing, (\ref{line1}) holds and
\begin{align}\label{line14}
\sum^\infty_{n=1} \frac{p^2 (n) \alpha (p (n))}{n} &< \infty,\\
\label{line15}
\sum^\infty_{n=1} \frac{p^3 (n)}{n^2} &< \infty.
\end{align}
\item $\left(X_n\right)_{n\in\mathds{N}}$ is near epoch dependent (with approximation constants $\left( a_l\right)_{l\in\N}$) on a stationary absolutely regular process $ \left(Z_n\right)_{n\in\mathds{Z}}$ with $\sum^\infty_{k = 0} k^2 ( a_k + \beta (k)) < \infty$.
\end{enumerate}
 Then 
\begin{align} 
\label{line9}\Var^* (\sqrt{kp} \bar{X}^*_{n,kp}) &\rightarrow \sigma^2 \qquad &\mbox { a.s.}\\ 
\label{line11}
 \sqrt{kp}(\bar{X}^*_{n,kp} - \bar{X}_{n,kp}) &\rightarrow N (0, \sigma^2) \text{ in distribution}
 \end{align}
almost surely.
\end{theorem}

\begin{remark}
The condition (\ref{line14}) implies $\sum^\infty_{n=1} \frac{\alpha (n)}{n} < \infty$. We can reformulate the first part of Theorem \ref{theo7} under above condition on mixing coefficients instead of conditions (\ref{line14}) and (\ref{line15}) claiming that there is a sequence $(p\left(n\right))$ that the statement of Proposition \ref{theo1} holds (although the Central Limit Theorem has not to hold), as it was done in Peligrad \cite{peli}. Under the conditions of the second part of Theorem \ref{theo7}, the Central Limit Theorem is true for $\bar{X}$, so that the statement of Proposition \ref{theo1} is also valid. 
\end{remark}

\subsection{Bootstrap for U-Statistics}

To bootstrap a $U$-statistic under dependence, one can apply the nonoverlapping block bootstrap and plug the observations $X_1^\star,\ldots,X_n^\star$ in:
\begin{multline*}
 U^{\star}_{n}\left(h\right)=\frac{2}{pk\left(pk-1\right)}\sum_{1\leq i<j\leq pk}h\left(X^{\star}_{i},X^{\star}_{j}\right)\\
=\theta+\frac{2}{pk}\sum_{i=1}^{pk}h_{1}\left(X_{i}^{\star}\right)+\frac{2}{pk\left(pk-1\right)}\sum_{1\leq i<j\leq pk}h_{2}\left(X^{\star}_{i},X^{\star}_{j}\right).
\end{multline*}

\begin{theorem}\label{theo8} Let $\left(X_n\right)_{n\in\mathds{N}}$ be a stationary process and $h$ a $P$-Lipschitz-continuous kernel with uniform $(2+\delta)$-moments for a $\delta>0$. Assume that $1/p(n)+p(n)/n\rightarrow0$ and one of the following two conditions holds:
\begin{enumerate}
\item $\left(X_n\right)_{n\in\mathds{N}}$ is strongly mixing, $E\left|X_1\right|^\gamma<\infty$ for a $\gamma>0$ and $\alpha\left(n\right)=O\left(n^{-\alpha}\right)$ for a $\alpha>\frac{3\gamma\delta+\delta+5\gamma+2}{2\gamma\delta}$.
\item $\left(X_n\right)_{n\in\mathds{N}}$ is near epoch dependent on an absolutely regular process with $\beta\left(n\right)=O\left(n^{-\beta}\right)$ for a $\beta>\frac{2+\delta}{\delta}$ and $a\left(n\right)=O\left(n^{-a}\right)$ for an $a>\frac{4+3\delta}{\delta}$. 
\end{enumerate}
then as $n\rightarrow\infty$
\begin{equation}
\Var^\star\left[\sqrt{pk}U_n^\star\left(h\right)\right]-\Var\left[\sqrt{n}U_n\left(h\right)\right]\rightarrow0,
\end{equation}
\begin{equation} \sup_{x\in\mathds{R}}\left|P^{\star}\left[\sqrt{pk}\left(U^{\star}_{n}\left(h\right)-E^{\star}\left[U^{\star}_{n}\left(h\right)\right]\right)\leq x\right]-P\left[\sqrt{n}\left(U_{n}\left(h\right)-\theta\right)\leq x\right]\right|\rightarrow0.
\end{equation}
in probability.
\end{theorem}

\begin{theorem}\label{theo9} Let $\left(X_n\right)_{n\in\mathds{N}}$ be a stationary process and $h$ a $P$-Lipschitz-continuous kernel with uniform $(2+\delta)$-moments for a $\delta>0$ and $E\left|h_1\left(X_1\right)\right|^{4+\delta}$. Assume that (\ref{line12}), (\ref{line13}), and one of the following two conditions holds:
\begin{enumerate}
\item $\left(X_n\right)_{n\in\mathds{N}}$ is strongly mixing, $E\left|X_1\right|^\gamma<\infty$ for a $\gamma>0$ and $\alpha\left(n\right)=O\left(n^{-\alpha}\right)$ for a $\alpha>\frac{3\gamma\delta+\delta+5\gamma+2}{2\gamma\delta}$.
\item $\left(X_n\right)_{n\in\mathds{N}}$ is near epoch dependent on an absolutely regular process with $\beta\left(n\right)=O\left(n^{-\beta}\right)$ for a $\beta>\frac{12+3\delta}{\delta}$ and $a\left(n\right)=O\left(n^{-a}\right)$ for an $a>\frac{9+3\delta}{\delta}$. 
\end{enumerate}
then a.s. as $n\rightarrow\infty$
\begin{equation}
 \label{line16}
\Var^\star\left[\sqrt{pk}U_n^\star\left(h\right)\right]-\Var\left[\sqrt{n}U_n\left(h\right)\right]\rightarrow0,
\end{equation}
\begin{equation} \sup_{x\in\mathds{R}}\left|P^{\star}\left[\sqrt{pk}\left(U^{\star}_{n}\left(h\right)-E^{\star}\left[U^{\star}_{n}\left(h\right)\right]\right)\leq x\right]-P\left[\sqrt{n}\left(U_{n}\left(h\right)-\theta\right)\leq x\right]\right|\rightarrow0.\label{line17}
\end{equation}
\end{theorem}

\section{Preliminary results}

\subsection{Central Limit Theorem, Moment and Maximum Inequalities for Partial Sums}

In this subsection we will give some known results which will be used in the next section in the proofs of the theorems. We set
\begin{equation*}
 S_n=\sum_{i=1}^{n}X_i
\end{equation*}
and assume w.l.o.g. that $EX_i=0$.

\begin{lemma}[Ibragimov \cite{ibra}]\label{lem1} Assume that one of the following conditions is satisfied:
\begin{itemize}
\item Let $\left(X_n\right)_{n\in\N}$ be a stationary sequence of strongly mixing r.v.'s with $EX_1 = \mu$ and $ (E\left| X_1\right|^{2 + \delta})^\frac{1}{2 + \delta} < \infty$ for some $ 0 < \delta \leq \infty.$ Assume that
\begin{equation*}
\sum^\infty_{k=1} \alpha^{\frac{\delta}{2 + \delta}} (k) < \infty.
\end{equation*}
\item Let $\left(X_n\right)_{n\in\N}$ be a stationary sequence that is near epoch dependent on a strongly mixing such that
\begin{align*}
\sum^\infty_{k=1} \alpha^{\frac{\delta}{2 + \delta}} (k) &< \infty\\
\sum_{k=1}^\infty \left(E\left|X_0-E[X_0|Z_{-l},\ldots,Z_{l}]\right|^{\frac{2+\delta}{1+\delta}}\right)^{\frac{1+\delta}{2+\delta}}&< \infty
\end{align*}
 and $ (E\left| X_1\right|^{2 + \delta})^\frac{1}{2 + \delta} < \infty$ for some $ 0 < \delta \leq \infty.$
\end{itemize}
Then $\sigma^2 = \Var X_1 + 2 \sum_{k=2}^\infty \Cov (X, X)<\infty$ and $ \frac{\Var S_n}{n} \rightarrow \sigma^2$. If in addition $\sigma^2 > 0$, then
\begin{equation*}
 n^{1/2} (\bar{X}_n - \mu) \rightarrow N (0, \sigma^2)\quad \text{ in distribution}.
\end{equation*}

\end{lemma}

\begin{lemma}[Shao \cite{sha2}]\label{lem2} Let $\left(\xi_n\right)_{n\in\N}$ be a strongly mixing sequence of r. v.'s with $E \xi_i = 0$ and $(E \left|\xi_i\right|^s)^{1/s} \leq \mathcal{D}_n$ for $ 1 \leq i \leq n$ and for some $1 < s \leq \infty.$ Assume that
\[ \alpha (i) \leq C_0 i ^{- \theta} \mbox { for some } {C_0 > 1} \mbox { and } \theta > 0.
\] 
Then there exists a constant $K=K\left(C_0, \theta, s\right)$, such that for any $x \geq K {\mathcal D}_n n^{1/2} \log n$
\begin{equation*}
P (\max_{i \leq n} \left| \sum^i_{j=1} \xi_j\right| \geq x)\leq K n ( \frac{\mathcal {D}_n} {x})^\frac{s(\theta + 1)} {s + \theta} (\log \frac{x}{{\mathcal D}_n})^\theta.
\end{equation*}
\end{lemma}

\begin{lemma}[Shao and Yu \cite{shao}]\label{lem2b} Let $\left(\xi_n\right)_{n\in\N}$ be a sequence of r. v.'s with $E \xi_i = 0$. Assume that there is a constant $C>0$ such that for any $n\geq 1$
\begin{equation*}
\sup_{k\in\N}E\left(\sum_{i=k+1}^{k+n}\xi_i\right)^2\leq Cn.
\end{equation*}
Then
\begin{equation*}
\lim_{n\rightarrow\infty}\frac{1}{\sqrt{n}\log^2 n}\sum_{i=1}^n \xi_i=0
\end{equation*}
almost surely.
\end{lemma}

\begin{lemma}[Yokoyama \cite{yoko}]\label{lem3} Let $\left(X_n\right)_{n\in\N}$ be a stationary strongly mixing sequence of r.v.'s with $E X_1 = \mu$ and $(E \left|X_1\right|^{2 + \delta})^{\frac{1}{2 + \delta}} < \infty$ for some $ 0 < \delta \leq \infty$ suppose that $2\leq s < 2 + \delta$ and
\[
\sum^\infty_{n =1} n^{\frac{s}{2} - 1} \left(\alpha (n)\right)^{(2 + \delta - s)/ (2 + \delta)} < \infty.
\]
Then there exists a constant $C$ depending only on $s, \delta$ and the mixing coefficients $\left(\alpha(n)\right)_{n\in\N}$ such that
\begin{equation*}
E \left| \sum^n_{i =1} (X_i - \mu)\right|^s \leq C n^{s/2} (E\left|X_1 \right|^{2 + \delta})^{\frac{s}{2 + \delta}}.
\end{equation*}
\end{lemma}

\begin{lemma}[Rio \cite{rio}, Peligrad \cite{peli}]\label{lem4} Let $\left(X_n\right)_{n\in\mathds{N}}$ be a strongly mixing sequence of r. v.'s with $E X_i = 0$ and $\left| X_i \right| \leq C$ a.s. Then there is a universal constant $K$ such that for every $x > 0$ and $n \geq 1$
\begin{equation*}
P \left(\max_{1 \leq i \leq n} \left| S_i \right| > x\right) \leq K x^{-2} \left( \sum^n_{i =1} E X^2_i + C^2 \cdot n \sum^n_{i =1} \alpha (i)\right).
\end{equation*}
\end{lemma}

\begin{lemma}[Borovkova et al. \cite{boro}]\label{lem5} Let $\left(X_n\right)_{n\in\mathds{N}}$ be a stationary sequence which is near epoch dependent (with approximation constants $\left(a_k\right)_{k\in\N}$) on an absolutely regular process $\left(Z_n\right)_{n\in\mathds{Z}}$ with mixing coefficients $\left( \beta (k)\right)_{k\in\N }$. Suppose that one of the following two conditions holds
\begin{enumerate}
 \item $ E X_0 = 0, E \left| X_0 \right|^{4 + \delta} < \infty,$ \\
$\sum^\infty_{k = 1} k^2 ( a_k^{\frac{\delta}{3 + \delta}} + ( \beta ( k ))^{\frac{\delta}{4 + \delta}}) < \infty\ \mbox { for some } \delta > 0.$
\item $ X_0 \mbox { is bounded a.s.}, EX_0 = 0,$\\
$\sum^\infty_{k =1} k^2 ( a_k + \beta (k)) < \infty.$
\end{enumerate}
Then $ \sigma^2 = EX^2_0 + 2 \sum^\infty_{k = 1} EX_0 X_k<\infty$ and in the case $ \sigma^2 > 0$ we have
\[
\frac{1}{\sqrt{n}} S_n \rightarrow N ( 0, \sigma^2)\quad \mbox { in distribution as } n \rightarrow \infty.
\]
\end{lemma}

\begin{lemma}[Borovkova et al. \cite{boro}]\label{lem6} Let $\left(X_n\right)_{n\in\mathds{N}}$ be a stationary sequence which is near epoch dependent (with approximation constants $\left(a_k\right)_{k\in\N}$) on an absolutely regular process $\left(Z_n\right)_{n\in\mathds{Z}}$ with mixing coefficients $\left(\beta(k)\right)_{k\in\mathds{N}}$. Assume that one of the conditions of Lemma \ref{lem5} holds. Then there exists a constant $C$ such that
\[
E S_n^4 \leq C n^2.
\]
\end{lemma}

\begin{lemma}[Borovkova et al. \cite{boro}]\label{lem7} Let $\left(X_n\right)_{n\in\mathds{N}}$ be a stationary sequence which is near epoch dependent (with approximation constants $\left(a_k\right)_{k\in\N}$) on an absolutely regular process $\left(Z_n\right)_{n\in\mathds{Z}}$ with mixing coefficients $ \left(\beta(k)\right)_{k\in\mathds{N}}.$ Assume that $ EX_0 = 0$ and one of the following two conditions holds:
\begin{enumerate}
 \item $X_0 \mbox { is bounded a. s. and } \sum^\infty_{ k = 0} (a_k + \beta (k)) < \infty,$
 \item $E \left|X_0\right|^{2 + \delta} < \infty \mbox { and } \sum^\infty_{k = 0} ( a_k^{\frac{\delta}{1 + \delta}} + ( \beta (k))^{\frac{\delta}{2 + \delta}} ) < \infty.$
\end{enumerate}
 Then there exists a constant $C$ such that
 \[
 E S_n^2 \leq C n.
 \]
 \end{lemma}
 
\begin{lemma}[Borovkova et al. \cite{boro}]\label{lem8} Let $\left(X_n\right)_{n\in\mathds{N}}$ be a stationary sequence which is near epoch dependent (with approximation constants $\left(a_k\right)_{k\in\N}$) on an absolutely regular process $\left(Z_n\right)_{n\in\mathds{Z}}$ with mixing coefficients $\left(\beta(k)\right)_{k\in\mathds{N}}$.\\
 Then
\begin{enumerate}
 \item if $\left| X_0\right| \leq M$ a. s. for all non-negative integers $i \leq j < k \leq l$, we have
\begin{multline*}
 \left| E (X_i X_j X_k X_l) - E (X_i X_j) E ( X_k X_l)\right|\leq \\
 \left( 4 ( \beta ( [ \frac{k-j}{3}]))^{\frac{\delta}{2 + \delta}} (E \left| X_0 \right|^{2 + \delta})^{\frac{2}{2 + \delta}} + 8 (a_{[\frac{k - j}{3}]})^{\frac{\delta}{1 + \delta}} ( E \left| X_0\right|^{2 + \delta})^{\frac{1}{1 + \delta}}\right)\cdot M^2.
\end{multline*}
\item if $E \left| X_0 \right|^{4 + \delta} < \infty$ for all non-negative $i \leq j < k \leq l$, we have
\begin{multline*}
\left| E (X_i X_j X_k X_l) - E (X_i X_j) E ( X_k X_l) \right|\\
\leq 4( \beta ( [ \frac{k - j}{3}]))^{\frac{\delta}{4 + \delta}} (E \left| X_0 \right|^{4 + \delta})^{\frac{4}{4 + \delta}} + 8 ( a_{[\frac{k - j}{3}]})^{\frac{\delta}{3 + \delta}} ( E \left| X_0 \right|^{4 + \delta})^{\frac{3}{3 + \delta}}.
\end{multline*}
\end{enumerate}

\end{lemma}

\subsection{Moment Inequalities for U-Statistics}

To control the moments of the degenerate part of a $U$-statistics, we need bounds for the covariance of $h_2$. Recall that $h_2$ is defined as
\begin{equation*}
h_2(x,y):=h(x,y) - h_1(x) -h_1(y) -\theta.
\end{equation*}

\begin{lemma}[Dehling, Wendler \cite{deh2}]\label{lem12}
Let $\left(X_n\right)_{n\in\mathds{N}}$ be a stationary process and $h$ a $P$-Lipschitz-continuous kernel with uniform $(2+\delta)$-moments for a $\delta>0$. Let be $\tau\geq0$ such that one of the following three conditions holds:
\begin{enumerate}
\item $\left(X_n\right)_{n\in\mathds{N}}$ is strongly mixing, $E\left|X_1\right|^\gamma<\infty$ for a $\gamma>0$ and $\sum_{k=0}^{n}k\alpha^{\frac{2\gamma\delta}{\gamma\delta+\delta+5\gamma+2}}\left(k\right)=O\left(n^\tau\right)$.
\item $\left(X_n\right)_{n\in\mathds{N}}$ is a near epoch dependent on an absolutely regular process and for $A_L=\sqrt{2\sum_{i=L}^{\infty}a_i}:$\\ $\sum_{k=0}^{n}k\left(\beta^{\frac{\delta}{2+\delta}}\left(k\right)+A^{\frac{\delta}{2+\delta}}_k\right)=O\left(n^\tau\right)$.
\end{enumerate}
Then:
\begin{equation*}
\sum_{i_{1},i_{2},i_{3},i_{4}=1}^{n}\left|E\left[h_{2}\left(X_{i_{1}},X_{i_{2}}\right)h_{2}\left(X_{i_{3}},X_{i_{4}}\right)\right]\right|=O\left(n^{2+\tau}\right).
\end{equation*}
\end{lemma}

If $(X_n)_{n\in\N}$ is near epoch dependent, it is not clear that the same holds for $(h_1(X_n))_{n\in\N}$. The following Lemma is very similar to Proposition 2.11 of Borovkova et al. \cite{boro} and gives an answer:

\begin{lemma}\label{lem14}Let $\left(X_n\right)_{n\in\mathds{N}}$ be $L^1$ near epoch dependent on the process $(Z_n)_{n\in\mathds{Z}}$ with approximation constants $(a_l)_{l\in\mathds{N}}$ and $g$ satisfy the variation condition with constant $L$. 
\begin{enumerate}
\item If $E|g(X_1)|^{1+\delta}<\infty$, then $(g(X_n))_{n\in\N}$ is near epoch dependent on $(Z_n)_{n\in\mathds{Z}}$ with approximation constants
\begin{equation*}
a'_l=(L+2^{\frac{1+2\delta}{1+\delta}}\|g(X_1)\|_{1+\delta})a_{l}^{\frac{1+\delta}{1+2\delta}}.
\end{equation*}
\item If $g(X_1)$ is bounded, then $(g(X_n))_{n\in\N}$ is near epoch dependent on $(Z_n)_{n\in\mathds{Z}}$ with approximation constants
\begin{equation*}
a'_l=(L+2\|g(X_1)\|_{\infty})a_{l}^{\frac{1}{2}}.
\end{equation*}
\end{enumerate}
\end{lemma}

\subsection{General Propositions for the Bootstrap of Stationary Sequences}

The proofs of Propositions \ref{theo1} and \ref{theo3} are mainly based on the methods developed in Peligrad \cite{peli} and Shao, Yu \cite{shao}. We will give full proofs for completeness.

\begin{proposition}\label{theo1} Let $\left\{X_i, i\geq 1\right\}$ be a stationary sequence of r.v'.s such that $EX_1 =\mu$ and $\Var X_1 <\infty.$ Assume that the following conditions hold
\begin{align}\label{line1}
&\Var \quad n^{1/2} (\bar{X}_n - \mu) \rightarrow \sigma^2 > 0 ,\\
\label{line2}
&n^{1/2} (\bar{X}_n - \mu) \rightarrow N (0, \sigma^2)\ \text{ in distribution},\displaybreak[0]\\
\label{line3}
&p^{1/2} (\bar{X}_{n,kp}- \mu) \rightarrow 0 \qquad \mbox{ a.s. },\displaybreak[0]\\
\label{line4}
&\frac{1}{kp} \sum^k_{i = 1}\left[ \left(\sum_{j \in B_i} (X_j - \mu)\right)^2 - E \left(\sum_{j \in B_i} (X_j - \mu)\right)^2\right] \rightarrow 0 \qquad \mbox{ a.s. },\displaybreak[0]\\
\label{line5}
&\frac{1}{kp} \sum^k_{i = 1} \left( \sum_{j \in B_i} (X_j - \mu) \right)^2 \mathds{1}_{\left\{\left|\sum_{j \in B_i} (X_j - \mu)\right|^2> \epsilon kp\right\}} \rightarrow 0 \qquad \mbox{ a.s. }
\end{align}
for any $\epsilon > 0.$ Then the following takes place as $n \rightarrow \infty $
\begin{align*}
\Var^* (\sqrt{kp} \bar{X}^*_{n,kp}) \rightarrow \sigma^2 \qquad &\mbox { a.s.}\\
 \label{line10} \sup_{x\in\R} \left| P^*\left(\sqrt{kp} (\bar{X}_{n,kp}^* - \bar{X}_{n,kp}) \leq x \right) -P \left( \sqrt{n} (\bar{X}_n - \mu) \leq x\right)\right| \rightarrow 0 \qquad &\mbox{ a.s. }
\end{align*}
\end{proposition}

\begin{proof} We note that
\[
\sqrt{k p} ( \bar {X}_{n, kp}^{\star} - \bar{X}_{n, kp}) = \frac{1}{\sqrt{k}} \sum^k_{i =1} Z_{i, n}^*
\]
where $ Z_{i, n}^* = \frac{1}{\sqrt{p}} \sum_{j \in B_i} (X_j^* - \bar{X}_{n, kp})$,\ $i =1, \ldots, k$ are i.i.d. r.v.'s (conditionally on $\left(X_n\right)_{n\in\N}$). By simple calculations we have
\begin{multline*}
E^* Z_{1, n}^{* 2} = \frac{1}{kp} \sum^k_{i =1} (\sum_{j \in B_i} (X_j - \bar{X}_{n, kp}))^2= \frac{1}{kp} \sum^k_{i =1} (\sum_{j \in B_i} (X_j - \mu))^2 - p (\bar{X}_{n, kp} - \mu)^2 \\
= \frac{1}{kp} \sum^k_{i =1}\left( ( \sum_{j \in B_i} (X_j - \mu))^2 - E (\sum_{j \in B_1} (X_j - \mu))^2\right)+ \frac{1}{p} \operatorname{Var} S_p - p(\bar{X}_{n, kp} - \mu)^2.
\end{multline*}
Conditions (\ref{line1}), (\ref{line3}) and (\ref{line4}) imply that a.s. as $n \rightarrow \infty$
\begin{equation}\label{line18}
E^* Z^{*2}_{1,n} \rightarrow \sigma^2,\ \text{ and consequently } \Var^\star\left[\sqrt{kp}X^\star_{n,kp}\right]\rightarrow \sigma^2.
\end{equation}
For any $ \epsilon > 0$, we have
\begin{align*}
E^* Z_{i,n}^{* 2} \mathds{1}_{\left\{Z_{i,n}^2 > \epsilon k\right\}}=& \frac{1}{kp} \sum^k_{i =1} ( \sum_{j \in B_i} ( X_j - \bar{X}_{n, kp}))^2 \mathds{1}_{\left\{\left| \sum_{j \in B_i} (X_j - \bar{X}_{n, kp})\right|^2 > k p \epsilon\right\}}\displaybreak[0]\\
\leq& \frac{4}{kp} \sum^k_{i =1} ( \sum_{j \in B_i} ( X_j - \mu))^2 \mathds{1}_{\left\{(\sum_{j \in B_i} (X_j - \mu ))^2> \frac{\epsilon k p}{4}\right\}}\\
&+ \frac{4}{kp} \sum^k_{i =1} p^2 ( \bar{X}_{n, kp} - \mu)^2 \mathds{1}_{\left\{p ( \bar{X}_{n, kp} - \mu)^2 > \frac{\epsilon k p }{4}\right\}}\displaybreak[0]\\
\leq& \frac{4}{kp} \sum^k_{i =1} (\sum_{j \in B_i} (X_j - \mu))^2 \mathds{1}_{\left\{(\sum_{j \in B_i} ( X_j - \mu))^2 > \frac{\epsilon kp}{4}\right\}}+ 4 p (\bar{X}_{n, kp} - \mu)^2.
\end{align*}
Now lines (\ref{line3}) and (\ref{line5}) imply that
\begin{equation}\label{line19}
E^*Z_{1,n}^{*2} \mathds{1}_{\left\{(Z_{1,n}^{*2} > \epsilon k\right\}} \rightarrow 0 \mbox{ a. s. as } n \rightarrow \infty
\end{equation}
what means that $Z^*_{i,n}, i = 1,2 \ldots$ satisfies the Lindeberg condition. Thus lines (\ref{line18}) and (\ref{line19}) imply the statement of the theorem.
\end{proof}

\begin{proposition}\label{theo2}Let $ \left(X_n\right)_{n\in\mathds{N}}$ be a stationary sequence of r.v.'s. with $ E X_1 = \mu, \Var X_1 < \infty.$ Assume that conditions (\ref{line1}), (\ref{line2}), (\ref{line4}) and for each fixed $x\in\R$
\begin{equation}\label{line6}
 \frac{1}{kp} \sum^k_{k=1} \left(\mathds{1}_{\left\{\frac{1}{\sqrt{p}} \sum_{j \in B_i} (X_j - \mu) \leq x \right\}} - P \left(\frac{1}{\sqrt{p}} \sum^p_{i = 1} (X_i - \mu) \leq x\right)\right) \rightarrow 0 \qquad \mbox { a.s. }
 \end{equation}
hold. Then the statement of Proposition \ref{theo1} remains true.
\end{proposition}

\begin{proof}
We define
\[
\widetilde{F}_n (x) = \frac{1}{k}\sum^k_{i =1} \mathds{1}_{\left\{\frac{1}{p} \sum^{ip}_{j =(i - 1) p + 1} (X_j^* - \mu) \leq x\right\}}.
\]
It is easy to see that the r.v.'s
\[
\left\{ \frac{1}{\sqrt{p}} \sum^{ip}_{j = (i - 1) p + 1} (X_j^* - \mu)\right\}, i = 1, \ldots , k
\]
are i.i.d. with distribution function $\widetilde{F}_n (x)$. We denote by $\widetilde{F}_n^{(m)}$ the distribution function of
\[
(kp)^{1/2} ( \bar{X}^*_{n,kp} - \bar{X}_{n,kp}) = \frac{1}{\sqrt{k}} \sum^k_{i =1} \left(\frac{1}{\sqrt{p}} \sum^{ip}_{j = (i - 1) p +1} ((X_j^* - \mu) -
E^* (X_j^* - \mu))\right).
\] 
Note that
\begin{align*}
\int x^2 d \widetilde{F}_n (x)
=& \frac{1}{kp} \sum^k_{i =1} \left((\sum_{j \in B_i} (X_j - \mu))^2 - E ( \sum^p_{j =1} (X_j - \mu))^2+ E ( \sum^p_{j=1} (X_j - \mu))^2\right)\\
=&\frac{1}{kp} \sum^k_{i =1} \left((\sum_{j \in B_i} (X_j - \mu))^2 - E ( \sum^p_{j =1} (X_j - \mu))^2 \right)+ \frac{1}{p} \Var S_p.
\end{align*}

From the lines (\ref{line1}) and (\ref{line4}) we have
\[
\int x^2 d \widetilde{F}_n (x) \rightarrow \sigma^{2} \qquad \mbox{ a.s. as } n \rightarrow \infty.
\]
Now conditions (\ref{line2}) and (\ref{line6}) imply
\[
\widetilde{F}_n (x) \rightarrow N (0, \sigma^2) \mbox { a.s. }
\]
The rest of the proof is the same as in the proof of Theorem 2.2 of Shao and Yu \cite{shao}.
\end{proof}

The following proposition is analogue of Proposition 3.1 of Peligrad \cite{peli} (In this proposition we assume that $\left(x_n\right)_{n\in\N}$ is a fixed realization of $\left(X_n\right)_{n\in\mathds{N}})$.

\begin{proposition} Let $\left(x_n\right)_{n\in\N}$ be a bounded sequence of real numbers. For each $n,$ let $T_{n 1}, T_{n 2}, \ldots T_{n k}$ be independent r.v.'s uniformly distributed on $\{1, 2, \ldots, k\}$ . Assume that condition (\ref{line7}) and
\begin{equation*}
V_n = \frac{p}{k} \sum^k_{i =1} (\bar{x}_{p i} - \bar{x}_{n, kp})^2 \rightarrow \sigma^2>0 
\end{equation*}
hold, where $\bar{x}_{pi} = \frac{1}{p} \sum_{j \in B_i} x_j$, $\bar{x}_{n,pk}=\frac{1}{kp}\sum_{i=1}^{kp}x_{i}$.
Then 
\[
\sqrt{kp} (\bar{X}^*_{n,kp} - \bar{x}_{n,kp}) = \sqrt{k p} (\frac{1}{k} \sum^k_{j = 1} \sum^k_{i =1} \mathds{1}_{\left\{T_{n j} = i\right\}} \bar{x}_{pi} - \bar{x}_{n, kp}) \rightarrow N (0, \sigma^2)
\]
in distribution.
\end{proposition}
\begin{proof} We have
\[
\Var \bar{X}^*_{n,kp} = \frac{1}{k^2} \sum^k_{i =1} (\bar{x}_{pi}-\bar{x}_{n,kp})^2
\]
and by line (\ref{line2}) obtain
\[
\Var (\sqrt{kp} \bar{X}^*_{n,kp}) = \frac{p}{k} \sum^k_{i =1} (\bar{x}_{pi} - \bar{x}_{n,kp})^2 \rightarrow \sigma^2.
\]
Note that
\begin{multline*}
 \sqrt{kp} (\bar{X}^*_{n,kp} - \bar{x}_{n,kp}) = \sqrt{kp} ( \frac{1}{k} \sum^k_{j =1} \sum^k_{i =1} \mathds{1}_{\left\{ T_{nj} = i\right\}} \bar{x}_{pi} - \bar{x}_{n,kp})\\
= \sqrt{kp} (\frac{1}{k} \sum^k_{j =1} \sum^k_{i = 1} \mathds{1}_{\left\{ T_{nj} = i\right\}} (\bar{x}_{pi} - \bar{x}_{n,kp}))
= \sum^k_{j=i} U_{nj},
\end{multline*}
where
\[
 U_{nj} = \frac{\sqrt{p}}{\sqrt{k}} \sum^k_{i=1} \mathds{1}_{\left\{T_{nj} = i\right\}} ( \bar{x}_{pi} - \bar{x}_{n,kp}).
\]
In our case Lindeberg condition holds if
\[
E \max_{1 \leq j \leq k} U^2_{nj} \rightarrow 0 \qquad \mbox{ as } n \rightarrow \infty.
\]
Taking into account that r. v.'s are bounded we have
\[
\left| U_{nj} \right| \leq \sqrt{\frac{p}{k}} \max_{i \leq i \leq k} \left| \bar{x}_{pi} - \bar{x}_{n, kp} \right| = O (\sqrt{\frac{p}{k}}).
\]
and
\[
\left| U_{nj} \right|^2 = O (\frac{p}{k}) = O (\frac{p^2}{n})
\]
Now the condition (\ref{line7}) implies the statement of the proposition.
\end{proof}

\begin{proposition}\label{theo3}Let $ \left(X_n\right)_{n\in\mathds{N}}$ be a stationary sequence of bounded almost surely r.v.'s with $EX_1 = \mu.$ Assume that conditions (\ref{line1}), (\ref{line3}) (\ref{line4}) and following conditions hold
 \begin{equation}\label{line7}
 \frac{p^2}{n}\rightarrow 0 \qquad \qquad \mbox{ as } n \rightarrow \infty.
 \end{equation}
 Then line (\ref{line9}) remains true and
 \begin{equation}
 \sqrt{kp}(\bar{X}^*_{n,kp} - \bar{X}_{n,kp}) \rightarrow N (0, \sigma^2) \text{ in distribution}.
 \end{equation}
\end{proposition}

\begin{proof}
The proof is based on the previous proposition, so we have to show that
\[
V_n = \frac{p}{k} \sum^k_{i =1} (\bar {X}_{pi} - \bar{X}_{n,kp})^2 \rightarrow \sigma^2 \mbox{ a.s. as } n \rightarrow \infty.
\]
W.l.o.g. assume that $\mu = EX_1 = 0$. Set $ S_{pi} = \sum_{j \in B_i} X_j.$ Then
\begin{align*}
 V_n &= \frac{p}{k} \sum^k_ {i =1} \bar{X}^2_{pi} - p \bar{X}^2_{n,kp} = \frac{p}{k} \sum^k_{i =1} \frac{S_{pi}^2}{p^2} - p \bar{X}^2_{n,kp}\\
&=\frac{1}{k} \sum^k_{i =1} \frac{S^2_{pi} - E S_{pi}^2}{p} + \frac{\Var S_p}{p} - p \bar{X}^2_{n,kp}.
\end{align*}

Conditions (\ref{line1}), (\ref{line3}) and (\ref{line4}) imply
\[
V_n \rightarrow \sigma^2 \mbox{ a.s. as } n \rightarrow \infty.
\]
This completes the proof of Proposition \ref{theo3}.
\end{proof}

\section{Proofs of the theorems}

\subsection{Bootstrap for the sample mean under dependence conditions}

\begin{proof}{Proof of Theorem \ref{theo5}} Under condition 1 (strong mixing), this theorem can be found in the book of Lahiri \cite{lahi}, Theorems 3.1 and 3.2. For the second condition (near epoch dependence), we will first show that the central limit theorem holds. By our assumption, we have that
\begin{multline*}
E\left|X_0-E[X_0|Z_{-l},\ldots,Z_{l}]\right|^{\frac{2+\delta}{1+\delta}}\\
\shoveleft= E\left[\left|X_0-E[X_0|Z_{-l},\ldots,Z_{l}]\right|^{\frac{2+\delta}{1+\delta}}\mathds{1}_{\{|X_0-E[X_0|Z_{-l},\ldots,Z_{l}]|\leq a_l^{-\frac{1}{1+\delta}}\}}\right]\\
+E\left[\left|X_0-E[X_0|Z_{-l},\ldots,Z_{l}]\right|^{\frac{2+\delta}{1+\delta}}\mathds{1}_{\{|X_0-E[X_0|Z_{-l},\ldots,Z_{l}]|> a_l^{-\frac{1}{1+\delta}}\}}\right]\\
\shoveleft\leq a_l^{-\frac{1}{(1+\delta)^2}}E\left|X_0-E[X_0|Z_{-l},\ldots,Z_{l}]\right|\\
+a_l^{\frac{\delta(2+\delta)}{(1+\delta)^2}}E\left|X_0-E[X_0|Z_{-l},\ldots,Z_{l}]\right|^{2+\delta}\leq C a_l^{\frac{\delta(2+\delta)}{(1+\delta)^2}}.
\end{multline*}
So we have that
\begin{equation*}
\sum_{k=1}^\infty \left(E\left|X_0-E[X_0|Z_{-l},\ldots,Z_{l}]\right|^{\frac{2+\delta}{1+\delta}}\right)^{\frac{1+\delta}{2+\delta}}\leq C\sum_{k=1}^\infty a_l^{\frac{\delta}{1+\delta}}<\infty,
\end{equation*}
and by Lemma \ref{lem1} the partial sum is asymptotically normal. We write $X_n=X_n(\nu)+X_n(\bar{\nu})$ with
\begin{equation*}
X_n(\nu)=E[X_n|Z_{n-\nu},\ldots,Z_{n+\nu}]
\end{equation*}
and $X_n(\bar{\nu})=X_n-X_n(\nu)$. Note that the sequence $(X_n(\nu))_{n\in\N}$ is strongly mixing with mixing coefficients $\alpha_X(t)=\alpha(i-2\nu)$. Then for the bootstrap version of the sample mean, we have that
\begin{multline*}
\sqrt{kp}\left(\bar{X}^\star_{n,kp}-\bar{X}_{n,kp}\right)\\
=\sqrt{kp}\left(\bar{X}^\star_{n,kp}(\nu)-\bar{X}_{n,kp}(\nu)\right)+\sqrt{kp}\left(\bar{X}^\star_{n,kp}(\bar{\nu})-\bar{X}_{n,kp}(\bar{\nu})\right)=I_n+II_n
\end{multline*}
where $I_n$ is constructed from the sampe $X_1(\nu),\ldots,X_n(\nu)$ and $II_2$ from the sample $X_1(\bar{\nu}),\ldots,X_n(\bar{\nu})$. The first part of this theorem implies that for any fixed $\nu\in\N$, the bootstap is valid for $\bar{X}_{n,kp}(\nu)$, meaning that
\begin{align}
\label{line25b}\Var^* (\sqrt{kp} \bar{X}^*_{n,kp}(\nu)) &\rightarrow \sigma^2(\nu)\\
\label{line26b} \sup_{x\in\R} \left| P^*\left(\sqrt{kp} (\bar{X}_{n,kp}^*(\nu) - \bar{X}_{n,kp})(\nu) \leq x \right) -\Phi(\frac{x}{\sqrt{\sigma^2(\nu)}})\right| &\rightarrow 0
\end{align}
in probability as $n\rightarrow\infty$ with $\sigma^2(\nu)=\lim_{n\rightarrow} \Var[\frac{1}{\sqrt{n}}\sum_{i=1}^nX_i(\nu)]$. Now we will estimate the variance of $II_n$
\begin{multline*}
\Var^\star[II_n]=\frac{1}{kp}\sum_{i=1}^k\left((\sum_{j^\in B_i}(X_i(\bar{\nu})-\mu))^2-E(\sum_{j^\in B_1}(X_i(\bar{\nu})-\mu))^2\right)\\
+\frac{1}{p}\Var[X_1(\bar{\nu})+\ldots+X_p(\bar{\nu})]-p\bar{X}^2_{n,kp}(\bar{\nu})
\end{multline*}
where $\mu=\mu(\nu)=EX_1(\bar{\nu})$. Observe that $E|X_i(\bar{\nu})|^{2+\delta}\leq E|X_i|^{2+\delta}<\infty$, so Lemma \ref{lem2b} implies that
\begin{equation*}
p\left(\bar{X}_{n,kp}(\bar{\nu})-\mu\right)^2\rightarrow0
\end{equation*}
a.s. as $n\rightarrow\infty$. For any $\epsilon>0$´, we have that by the Markov inequality
\begin{align*}
&P\left(\left|\frac{1}{kp}\sum_{i=1}^k\left((\sum_{j^\in B_i}(X_i(\bar{\nu})-\mu))^2-E(\sum_{j^\in B_1}(X_i(\bar{\nu})-\mu))^2\right)\right|>\epsilon\right)\\
\leq&\frac{1}{\epsilon kp}E\left|\frac{1}{kp}\sum_{i=1}^k\left((\sum_{j^\in B_i}(X_i(\bar{\nu})-\mu))^2-E(\sum_{j^\in B_1}(X_i(\bar{\nu})-\mu))^2\right)\right|\\
\leq&\frac{2}{\epsilon p}E\left(\sum_{j^\in B_i}(X_i(\bar{\nu})-\mu)\right)^2=\frac{2}{\epsilon p}\left(p\Var X_1(\bar{\nu})+2\sum_{j=2}^p(p-j)\Cov(X_1(\bar{\nu}),X_j(\bar{\nu}))\right)\\
&\leq \frac{2}{\epsilon}\left((2N+1)\Var(X_1(\bar{\nu}))+2C\sum_{j=N}^p\left(\|X_0-E[X_0|Z_{-\lfloor \frac{l}{l}\rfloor},\ldots,Z_{\lfloor \frac{j}{3}\rfloor}]\|_{\frac{2+\delta}{1+\delta}}+\alpha^{\frac{\delta}{2+\delta}}(\frac{j}{3})\right)\right)=:A
\end{align*}
where $1\leq N\leq p$. Now the conditions of the theorem imply that we can choose a fixed $N\in\N$ and $\nu\in\R$, such that
\begin{align*}
A&\leq\epsilon\\
\sup_{x\in\R}\|\Phi(\frac{x}{\sigma(\nu)})-\Phi(\frac{x}{\sigma})\|&\leq\epsilon\\
\frac{1}{p}\Var[X_1(\bar{\nu})+\ldots+X_p(\bar{\nu})]&\leq\epsilon.
\end{align*}
We note that for this $\nu$ the lines (\ref{line25b}) and (\ref{line26b}) hold, which together with these three inequalities imply the statement of the theorem.
\end{proof}

\begin{proof}[Proof of Theorem \ref{theo6}] The proof under the first assumption (strong mixing) is based on Proposition \ref{theo2}. Lemma \ref{lem1} implies that conditions (\ref{line1}) and (\ref{line2}) of Theorem \ref{theo2} are satisfied. It remains to prove lines (\ref{line4}) and (\ref{line6}). W.l.g. we can assume that $\mu = 0$ and we will prove
\[
\frac{1}{kp} \sum^k_{i =1} \left( ( \sum^p_{j =1} X_{(i - 1)p + j})^2 - E ( \sum^p_{i =1} X_i)^2 \right) \rightarrow 0 \mbox{ a.s. as } n \rightarrow \infty.
\]
By the Borel-Cantelli lemma, it suffices to show that for any $\epsilon>0$
\[
\sum^\infty_{l =1} P \left(\max_{2^l < n \leq 2^{l+1}}\left| \frac {1}{kp} \sum^k_{i=1} \left\{( \sum^p_{j =1} X_{(i -1)p + j} )^2 - E S_p^2 \right\}\right| > \epsilon \right)< \infty.
\]
Taking into account that $kp \backsim n$ and lines (\ref{line12}), (\ref{line13}) we have
\begin{align*}
I_l &:= P\left(\max_{2^l < n \leq 2^{l+1}} \left| \sum_{i =1}^{k(n)} \left((\sum^{p(2^l)}_{j =1} X_{(i - 1)p+ j})^2 - ES_p^2\right) \right| > \epsilon k (2^l) p (2^l) \right)\displaybreak[0]\\
&\leq P \left( \max_{n \leq 2^{l +1}} \left| \sum^{k (n)}_{i =1} \left(( \sum^{p (2^l)}_{j =1} X_{(i - 1)p+ j})^2 - ES^2_{p(2^l)} \right)\right| > \epsilon C 2^l\right)\displaybreak[0]\\
&\leq P \left(\max_{m \leq k (2^{l + 1})} \left| \sum^m_{i=1} \left( ( \sum^{p(2^l)}_{j =1} X_{(i-1)p+ j})^2 - ES^2_{p(2^l)}\right) \right| > \epsilon C 2^l \right).
\end{align*}
From Lemma \ref{lem3} we have for some $s >1$ that
\begin{equation}\label{line21}
(E \left| ( \sum^{p(2^l)}_{j =1} X_{(i - 1)p + j} )^2 - ES^2_{p(2^l)} \right|^s) ^{1/s} \leq C ( E \left| S_{p (2^l)} \right|^{2 s})^{1/s} \leq C p (2^l).
\end{equation}
Now using Lemma \ref{lem2} and taking into account line (\ref{line21}), we obtain
\begin{equation*}
I_l \leq C k (2^{l + 1}) ( \frac{p (2^l)}{2^l})^{\frac{s(r+1)}{s + r}} \log^r (\frac{2^l}{p (2^l)}) \leq C (\frac{p (2^l)}{2^l})^{\frac{(s-1)r}{s+r}} \log^r (\frac{2^l}{p (2^l)}).
\end{equation*}
From the condition (\ref{line12}), it follows that
\[
\sum^\infty_{l =1} P \left(\max_{2^l < n \leq 2^{l+1}}\left| \frac {1}{kp} \sum^k_{i=1} \left(( \sum^p_{j =1} X_{(i -1)p+ j} )^2 - E S_p^2 \right)\right| > \epsilon \right)\leq\sum^\infty_{l =1} I_l < \infty.
\]
It remains to prove (\ref{line6}), i.e.
\[
\frac{1}{kp} \sum^k_{i =1} \left(\mathds{1}_{\left\{\frac{1}{\sqrt{p}} \sum^p_{j =1} X_{(i -1)p+j} \leq x\right\}} - P \left(\frac{1}{\sqrt{p}} \sum^p_{i =1} X_i \leq x \right)\right)\ \rightarrow 0 \mbox { a.s. }
\]
Because of the Borel-Cantelli lemma, it suffices to show that
\begin{multline*}
\sum^\infty_{l = 1} P \left( \max_{2^l< n \leq 2^{l + 1}} \left | \sum^k_{i = 1} \left( \mathds{1}_{\left\{ \frac{1}{\sqrt{p}} \sum^p_{j = 1} X_{(i - 1)p + j} \leq x\right\}} \right.\right.\right.\\
\left.\left.\left. - P \left( \frac{1}{\sqrt{p}} \sum^p_{i = 1} X_i \leq x \right)\right) \right| > \epsilon k (2^l) p (2^l) \right) < \infty.
\end{multline*}
Using Lemma \ref{lem4} we conclude
\begin{align*}
II_l :=& P \left(\max_{2^l < n \leq 2^{l + 1}} \Bigg| \sum^{k\left(n\right)}_{i =1} \left(\mathds{1}_{\left\{\frac{1}{\sqrt{p (2^l)}} \sum^{p(2^l)}_{j =1} X_{(i - 1)p + j} \leq x\right\}}\right.\right.\\
 &\left.\left.- P( \frac{1}{\sqrt{p (2^l)}} \sum^{p(2^l)}_{i=1} X_i \leq x )\right) \Bigg| > \epsilon k (2^l) p (2^l) \right)\displaybreak[0]\\
 \leq& P \left(\max_{m \leq k(2^{l+1})} \Bigg| \sum^m_{i = 1}\left(\mathds{1}_{\left\{ \frac{1}{\sqrt{p (2^l)}} \sum^{p (2^l)}_{j =1} X_{(i - 1) p + j} \leq x\right\}}\right.\right.\\
 &-\left.\left. P (\frac{1}{\sqrt{p( (2^l)}} \sum^{p (2^l)}_{i =1} X_i \leq x)\right)\Bigg| > \epsilon k (2^l) p (2^l)\right)\\
 \leq& \frac {C \left(k( 2^{l + 1} )+ k (2^{l + 1}) \sum^{k(2^{l + 1})}_{i =1} \bar{\alpha} (i)\right)}{\epsilon^2 k^2 (2^l) p^2 (2^l)}.
\end{align*}
where $\bar{\alpha} (i) = \alpha \left((i - 1)p(2^l)+1\right)$. As $\sum_{i=1}^{\infty}\alpha\left(i\right)<\infty$:
\[
II_l \leq \frac{C k(2^{l+1})}{\epsilon^2 k^2 (2^l) p^2 (2^l)}.
\]
From condition (\ref{line12}) we get that
\[
\sum^\infty_{l = 1} II_l < \infty,
\]
so we have completed the proof under strong mixing. For the second assumption (near epoch dependence), we will check the assumptions of Proposition \ref{theo1}. Lemma \ref{lem5} implies the conditions (\ref{line1}) and (\ref{line2}). W.l.o.g. we will assume that $EX_1 = \mu = 0$. First we will prove line (\ref{line3}). Note that by stationarity and Lemma \ref{lem7} we have for any $ a \geq 1$ and some $ C > 0$
\[
E \left| \sum^a_{i =1} S_{p,i}\right| ^2 \leq C a p.
\]
Theorem $A$ of Serfling \cite{serf} implies that
\begin{equation}\label{line24}
E \left[\max_{1 \leq a \leq m} \left| \sum^a_{i =1} S_{p,i} \right|^2\right] \leq C m (\log 2 m)^2 p.
\end{equation}
In order to prove line (\ref{line3}) it suffices to show that for any $ \epsilon > 0$

\begin{equation}\label{line25}
\sum^\infty_{l = 1} P( \max_{2^l < n \leq 2^{l + 1}} \left| p^{1/2} \bar{X}_{n, kp} \right| > \epsilon ) < \infty.
\end{equation}
Keep in mind that $p=p(n)$ does not change for $2^l < n \leq 2^{l + 1}$ by condition (\ref{line13}). By Chebyshev inequality and (\ref{line24}), it follows that
\begin{align*}
&P ( \max_{2^l < n \leq 2^{l + 1}} \left| p^{1/2} \bar{X}_{n, kp} \right| > \epsilon )\leq P \left(\max_{2^l < n \leq 2^{l +1}} \left| \sum^k_{i= 1} S_{p,i} \right| > k (2^l) p^{1/2} (2^l) \epsilon\right)\displaybreak[0]\\
\leq & \frac {C \epsilon ( \max_{1 \leq j \leq k (2^{l + 1})} \left| \sum^k_{i = 1} S_{p,i} \right|^2)}{k^2 (2^l) p (2^l) \epsilon} \leq \frac { C k (2^{l + 1}) ( \log (2 \cdot k (2^{l + 1})))^2 p (2^l)}{k^2 (2^l) p (2^l) \epsilon}\\
\leq& \frac { C ( \log (2 \cdot k (2^{l + 1})))^2}{k (2^l) \epsilon}.
\end{align*}
The latter implies line (\ref{line25}) and hence (\ref{line3}) is proved, so we can proceed verifying line (\ref{line4}). First we will prove the existence of the constant $ C > 0$ such that
\begin{equation*}
E \left| \sum^m_{i =1} (S_{p,i}^2 - E S^2_p)\right|^2 \leq C m p^2
\end{equation*}
for $ m \geq 1$. Stationarity and Lemmas \ref{lem6}, \ref{lem7} and \ref{lem8} imply
\begin{align*}
 &E \left| \sum^m_{i =1} ( S_{p,i}^2 - ES_p^2)\right|^2 \leq 4E \left| \sum^m_{\substack{i =1\\i \text{ odd}}} ( S_{p,i}^2 - ES_p^2)\right|^2\\
 =& 4m E \left( S^2_p - E S^2_p\right)^2 + 8 \sum^{ m -1}_{ i = 3} ( m - i + 1) E (S^2_p - E S^2_p) \cdot (S_{p,i}^2 - E S^2_p)\displaybreak[0]\\
\leq & C m p^2 + Cm \sum^{m - 1}_{i = 3} \sum_{\substack{i_1, i_2 \in B_1\\j_1, j_2 \in B_2}} \left| E X_{i_1} X_{i_2} X_{j_1} X_{j_2} - E X_{i_1} X_{i_2} E X_{j_1} X_{j_2} \right| \\
\leq & C m p^2 + Cm \sum^{m - 1}_{i = 2} p^4 \left((a_{[\frac{(i - 2) p}{3}]})^{\frac{\delta}{3 + \delta}} + ( \beta ([\frac{( i - 2) p}{3}]))^{\frac{\delta}{4 + \delta}}\right)\\
\leq & C m p^2 + C m p^2 \sum^\infty_{k = 1} k^2 \left((a_{[\frac{k}{3}]})^{\frac{\delta}{3 + \delta}} + ( \beta ([\frac{k}{3}]))^{\frac{\delta}{4 + \delta}}\right)\leq C m p^2.
\end{align*}
Now again using Theorem $A$ of Serfling \cite{serf}, we obtain
\begin{equation}\label{line27}
E \left[\max_{ 1 \leq a \leq m} \left| \sum^a_{i =1} ( S^2_{p,i} - E S^2_{p,i})\right|^2\right] \leq C m (\log 2 m)^2 p^2.
\end{equation}
If we can prove
\begin{equation}\label{line28}
\sum^\infty_{l = 1} P ( \max_{2^l < n \leq 2^{l + 1}} \left| \sum^k_{i = 1} (S_{p,i}^2 - E S^2_p)\right|
 > \epsilon k p ) < \infty,
\end{equation}
line (\ref{line4}) follows by the Borel-Cantelli lemma. Using Chebyshev inequality and (\ref{line27}) we get
\begin{align*}
 &P \left( \max_{2^l < n \leq 2^{l + 1}} \left| \sum^k_{i = 1} ( S^2_{p,i} - E S^2_p) \right| > \epsilon k p \right)\\
\leq& P \left(\max_{1 \leq m \leq k(2^{l + 1})}\left| \sum^m_{i =1} ( S^2_{p,i} - E S^2_p) \right| > \epsilon k (2^l) p (2^l)\right)\\
\leq& \frac{C k (2^{l + 1}) (\log (2 k ( 2^{l + 1})))^2 \cdot p^2 (2^l)}{\epsilon k^2 (2^l) p^2 (2^l)} 
\leq \frac{C (\log ( 2 k (2^{l + 1})))^2} {\epsilon k (2^l)}.
\end{align*}
The latter implies line (\ref{line28}) and hence (\ref{line4}). It remains to check line (\ref{line5}). In order to do that it suffices to show that for any $\epsilon > 0$ and $\epsilon_1 > 0$ 
\begin{equation}\label{line23}
\sum^\infty_{l = 1} P\left( \max_{2^l < n \leq 2^{l+1}} \left| \frac{1}{kp} \sum^k_{i = 1} S^2_{p,i} \mathds{1}_{\left\{ \left| S_{p,i} \right|^2 > \epsilon k p\right\}} \right| > \epsilon_1\right) < \infty
\end{equation}
where $ S_{p,i} = \sum_{j \in B_i} X_j.$ Using Markov, H\"older, Chebyshev inequalities, line (\ref{line13}), and Lemma \ref{lem6} we obtain 
\begin{align*}
& P ( \max_{2^l < n \leq 2^{l + 1}} \left| \frac{1}{k p} \sum^k_{i = 1} S^2_{p,i} \mathds{1}_{\left\{\left| S_{p,i} \right|^2 > \epsilon k p \right\}} \right| > \epsilon_1) \\
\leq & P ( \max_{2^l < n \leq 2^{l + 1}} \left| \sum^k_{i = 1} S^2_{p,i} \mathds{1}_{\left\{\left| S_{p,i} \right|^2 > \epsilon k (2^l) p (2^l)\right\}} \right| > \epsilon_1 k (2^l) p (2^l))\displaybreak[0]\\
\leq &\frac{E (\max_{2^l < n \leq 2^{l + 1}} \left| \sum^k_{i = 1} S^2_{p,i} \mathds{1}_{\left\{\left| S_{p,i} \right|^2 > \epsilon k (2^l) p (2^l) \right\}}\right| } {\epsilon_1 k (2^l) p (2^l)}\displaybreak[0]\\
\leq & \frac{\sum^{k ( 2^{l +1})}_{i =1} E S^2_{p,i} \mathds{1}_{\left\{\left| S_{p,i} \right|^2 > \epsilon k (2^l) p (2^l)\right\}}}{\epsilon_1 k (2^l) p (2^l)}\leq \frac{k (2^{l + 1}) E S_p^2 \mathds{1}_{\left\{(\left| S_p \right|^2 > \epsilon k (2^l) p ( 2^l)\right\}}}{\epsilon_1 k (2^l) p (2^l)}\displaybreak[0]\\
\leq & \frac{C (E S^4_p)^{1/2} ( P ( \left| S_p \right|^2 > \epsilon k (2^l) p (2^l))^{1/2}}{\epsilon_1 p (2^l)}\leq \frac{C \quad E S^4_p} {\epsilon_1 \cdot \epsilon k (2^l) p^2 (2^l)}\\
 \leq& \frac{C \quad p^2 (2^l)}{\epsilon_1 \epsilon k (2^l) p^2 (2^l)} \leq \frac{ C}{\epsilon_1 \cdot \epsilon k (2^l)}.
\end{align*}
The latter implies line (\ref{line23}) and thus (\ref{line5}) is proved.
\end{proof}

\begin{proof}[Proof of Theorem \ref{theo7}] The proof is based on Proposition \ref{theo3}. The verification of lines (\ref{line1}),(\ref{line3}),(\ref{line4}) can be done along the lines of the proof above and is hence omitted.

\end{proof}

\subsection{Bootstrap for U-Statistics}

\begin{lemma}\label{lem13} Let $\left(X_n\right)_{n\in\mathds{N}}$ be a sequence of r.v.'s and $A\subset \left\{1,\ldots,n\right\}^4$. Then there is a constant $C$, such that:
\begin{multline*}
\left|EE^\star\left[\sum_{(i_{1},i_{2},i_{3},i_{4})\in A}h_{2}\left(X_{i_{1}}^\star,X_{i_{2}}^\star\right)h_{2}\left(X_{i_{3}}^\star,X_{i_{4}}^\star\right)\right]\right|\\
\leq C\sum_{i_{1},i_{2},i_{3},i_{4}=1}^{n}\left|E\left[h_{2}\left(X_{i_{1}},X_{i_{2}}\right)h_{2}\left(X_{i_{3}},X_{i_{4}}\right)\right]\right|,
\end{multline*}
where $h_2$ is defined by the Hoeffding-decomposition.
\end{lemma}

\begin{proof}
By triangle inequality:
\begin{multline*}
\left|EE^\star\left[\sum_{(i_{1},i_{2},i_{3},i_{4})\in A}h_{2}\left(X_{i_{1}}^\star,X_{i_{2}}^\star\right)h_{2}\left(X_{i_{3}}^\star,X_{i_{4}}^\star\right)\right]\right|\\
\leq \frac{1}{(pk)^{2}(pk-1)^{2}}\sum_{(i_{1},i_{2},i_{3},i_{4})\in A}\left|EE^{\star}\left[h\left(X_{i_{1}},X_{i_{2}}\right)h\left(X_{i_{3}},X_{i_{4}}\right)\right]\right|.
\end{multline*}
The bootstrapped expectation of $h_{2}\left(X_{i_{1}}^{\star},X_{i_{2}}^{\star}\right)h_{2}\left(X_{i_{3}}^{\star},X_{i_{4}}^{\star}\right)$ (conditionally on $\left(X_{n}\right)_{n\in\mathds{N}}$) depends on the way the indices $i_{1},i_{2},i_{3},i_{4}$ are allocated to the different blocks. First consider indices $i_{1},i_{2},i_{3},i_{4}$ lying in four different blocks $B_{j_{1}},B_{j_{2}},B_{j_{3}},B_{j_{4}}$ (therefore, $X_{i_{1}}^{\star},\ldots,X_{i_{4}}^{\star}$ are independent for fixed $\left(X_{n}\right)_{n\in\mathds{N}}$). From the construction of the bootstrap sample for any four different blocks $B_{j_{1}},B_{j_{2}},B_{j_{3}},B_{j_{4}}$:
\begin{align*}
&\left|E\left[E^{\star}\left[h_{2}\left(X_{i_{1}}^{\star},X_{i_{2}}^{\star}\right)h_{2}\left(X_{i_{3}}^{\star},X_{i_{4}}^{\star}\right)\right]\right]\right|\\
&=\big|E\big[\frac{1}{k^{4}}\sum_{\substack{1\leq i_1+k_1p\leq n\\1\leq i_2+k_2p\leq n\\1\leq i_3+k_3p\leq n\\1\leq i_4+k_4p\leq n}}h_{2}\left(X_{i_{1}+k_1p},X_{i_{2}+k_2p}\right)h_{2}\left(X_{i_{3}+k_3p},X_{i_{4}+k_4p}\right)\big]\big|\displaybreak[0]\\
&\leq\frac{1}{k^{4}}\sum_{\substack{1\leq i_1+k_1p\leq n\\1\leq i_2+k_2p\leq n\\1\leq i_3+k_3p\leq n\\1\leq i_4+k_4p\leq n}} \left|E\left[h_{2}\left(X_{i_{1}+k_1p},X_{i_{2}+k_2p}\right)h_{2}\left(X_{i_{3}+k_3p},X_{i_{4}+k_4p}\right)\right]\right|\displaybreak[0]\\
\Rightarrow&\sum_{\substack{(i_1,i_2,i_3,i_4)\\ \in (B_{j_{1}}\times B_{j_{2}}\times B_{j_{3}}\times B_{j_{4}})\cap A}}\left|EE^{\star}\left[h\left(X_{i_{1}},X_{i_{2}}\right)h\left(X_{i_{3}},X_{i_{4}}\right)\right]\right|\\
&\leq\frac{1}{k^4}\sum_{i_{1},i_{2},i_{3},i_{4}=1}^{n}\left|E\left[h\left(X_{i_{1}},X_{i_{2}}\right)h\left(X_{i_{3}},X_{i_{4}}\right)\right]\right|
\end{align*}
As there are less than $k^4$ possibilities to choose these four blocks, one gets:
\begin{multline*}
\sum_{\substack{(i_1,i_2,i_3,i_4)\in A\\\mbox{\tiny{4 diff. blocks}}}}\left|EE^{\star}\left[h\left(X_{i_{1}},X_{i_{2}}\right)h\left(X_{i_{3}},X_{i_{4}}\right)\right]\right|\\
\leq\sum_{i_{1},i_{2},i_{3},i_{4}=1}^{n}\left|E\left[h\left(X_{i_{1}},X_{i_{2}}\right)h\left(X_{i_{3}},X_{i_{4}}\right)\right]\right|.
\end{multline*}

As an example, let $i_{1}$ and $i_{2}$ now lie in the same block with $i_{2}-i_{i}=d>0$, while $i_{3}$, $i_{4}$ lie in two further blocks. $X_{i_{1}}^{\star}$ and $X_{i_{2}}^{\star}$ are dependent, the value of $X_{i_2}^{\star}$ is determined by the value of $X_{i_1}^{\star}$ (conditionally on $\left(X_n\right)_{n\in\mathds{N}}$). To repair this, add up the expected values for all $i_{2}$ such that $i_2$ in the same block as $i_{1}$ and take into account that there are at most $k^{3}$ possibilities for $i_{1},i_{3},i_{4}$:
\begin{align*}
&\left|E\left[E^{\star}\left[h_{2}\left(X_{i_{1}}^{\star},X_{i_{2}}^{\star}\right)h_{2}\left(X_{i_{3}}^{\star},X_{i_{4}}^{\star}\right)\right]\right]\right|\\
&\leq\frac{1}{k^{3}}\sum_{\substack{1\leq i_1+k_1p\leq n-d\\1\leq i_3+k_3p\leq n\\1\leq i_4+k_4p\leq n}}\left|E\left[h_{2}\left(X_{i_{1}},X_{i_{1}+d}\right)h_{2}\left(X_{i_{3}},X_{i_{4}}\right)\right]\right|\displaybreak[0]\\ 
\Rightarrow&\sum_{\substack{i_{2}\\(i_1,i_2,i_3,i_4)\in A\\i_2\mbox{\tiny{ in same block as }} i_1}}\left|E\left[E^{\star}\left[h_{2}\left(X_{i_{1}}^{\star},X_{i_{2}}^{\star}\right)h_{2}\left(X_{i_{3}}^{\star},X_{i_{4}}^{\star}\right)\right]\right]\right|\\
&\leq \frac{1}{k^{3}}\sum_{i_{1},i_{2},i_{3},i_{4}=1}^{n}\left|E\left[h_{2}\left(X_{i_{1}},X_{i_{2}}\right)h_{2}\left(X_{i_{3}},X_{i_{4}}\right)\right]\right|\\
 \Rightarrow&\sum_{\substack{(i_1,i_2,i_3,i_4)\in A\\i_2\mbox{\tiny{ in same block as }} i_1}}\left|E\left[E^{\star}\left[h_{2}\left(X_{i_{1}}^{\star},X_{i_{2}}^{\star}\right)h_{2}\left(X_{i_{3}}^{\star},X_{i_{4}}^{\star}\right)\right]\right]\right|\\
&\leq\sum_{i_{1},i_{2},i_{3},i_{4}=1}^{n}\left|E\left[h_{2}\left(X_{i_{1}},X_{i_{2}}\right)h_{2}\left(X_{i_{3}},X_{i_{4}}\right)\right]\right|
\end{align*}
When the indices are allocated to the blocks in another way, analogous arguments can be used, which completes the proof.
\end{proof}

\begin{proof}[Proof of Theorem \ref{theo8}] The proof under condition 1 (strong mixing) is similar to the proof of Theorem 2.2 of Dehling and Wendler \cite{deh2} and hence omitted. Under near epoch dependence, first note that by Lemma \ref{lem14} $\left(h_1(X_n)\right)_{n\in\N}$ is near epoch dependent with approximations constants $a'_k\leq Ca_k^{\frac{1+\delta}{1+2\delta}}\leq Ck^{-\frac{(4+3\delta)(1+\delta)}{\delta(1+2\delta)}}$, so $\sum_{k=1}^\infty a_k'^{\frac{\delta}{1+\delta}}<\infty$ and by Theorem \ref{theo5}
\begin{multline*}
 \sup_{x\in\mathds{R}}\left|P^{\star}\left[\frac{2}{\sqrt{pk}}\sum_{i=1}^{pk}\left(h_{1}\left(X_{i}^{\star}\right)-E^\star h_{1}\left(X_{i}^{\star}\right)\right)\leq x\right]-P\left[\frac{2}{\sqrt{n}}\sum_{i=1}^{n}h_{1}\left(X_{i}\right)\leq x\right]\right|\rightarrow0
\end{multline*}
in probability as $n\rightarrow\infty$. Furthermore, by Lemma \ref{lem12} and Lemma \ref{lem13}, we have that
\begin{align*}
E\left(\sqrt{n}U_n(h_2)\right)^2&\rightarrow0\\
EE^\star\left(\sqrt{n}U_n(h_2)\right)^2\leq C\frac{1}{n^3}\sum_{i_{1},i_{2},i_{3},i_{4}=1}^{n}\left|E\left[h_{2}\left(X_{i_{1}},X_{i_{2}}\right)h_{2}\left(X_{i_{3}},X_{i_{4}}\right)\right]\right|&\rightarrow0.
\end{align*}
The Chebyshev inequality completes the proof.

\end{proof}

\begin{proof}[Proof of Theorem \ref{theo9}] We first show that
\begin{equation*}
 P\left[\sqrt{pk}U_{n}^{\star}\left(h_2\right)\rightarrow 0\right]=E\left[P^\star\left[\sqrt{pk}U_{n}^{\star}\left(h_2\right)\rightarrow 0\right]\right]=1.
\end{equation*}
With Fubinis theorem, we will then conclude that
\begin{equation}
 P^\star\left[\sqrt{pk}U_{n}^{\star}\left(h_2\right)\rightarrow 0\right]=1\ \text{a.s.}\label{line33}
\end{equation}
We set
\begin{equation*}
 Q_n^\star=\sum_{1\leq i_1<i_2\leq pk}h_2\left(X^\star_{i_1},X^\star_{i_2}\right)\ \text { and }\ b_n=\frac{1}{\sqrt{pk}\left(pk-1\right)}.
\end{equation*}
With the method of subsequences, it suffices to show that
\begin{align}
b_{2^l}Q_{2^l}^{\star}\left(h_2\right)&\rightarrow0\ \text{a.s.,} \label{line29}\\
\max_{2^{l-1}\leq n<2^l}\left|b_nQ_n^{\star}-b_{2^{l-1}}Q_{2^{l-1}}^{\star}\right|&\rightarrow0\ \text{a.s.,}\label{line30}
\end{align}
as $l\rightarrow\infty$. By the conditions 1. or 2. of the theorem and Lemma \ref{lem12}, there exists a $\eta>0$ such that
\begin{equation}
\sum_{i_{1},i_{2},i_{3},i_{4}=1}^{n}\left|E\left[h_{2}\left(X_{i_{1}},X_{i_{2}}\right)h_{2}\left(X_{i_{3}},X_{i_{4}}\right)\right]\right|=O\left(n^{3-\eta}\right)\label{line31}.
\end{equation}
We use Chebyshev inequalitity and Lemma \ref{lem13} to prove line (\ref{line29}). For every $\epsilon>0$:
\begin{align*}
&\sum_{l=1}^{\infty}P\left[\left|b_{2^l}Q_{2^l}^\star\left(h_2\right)\right|>\epsilon\right]\leq\frac{1}{\epsilon^2}\sum_{l=1}^{\infty}b_{2^l}^2EE^\star\left[Q_{2^l}^{\star 2}\left(h_2\right)\right]\\
\leq& C\frac{1}{\epsilon^2}\sum_{l=1}^{\infty}b_{2^l}^2\sum_{i_1,i_2,i_3,i_4=1}^{2^{l}}\left|E\left[h_2\left(X_{i_1},X_{i_2}\right)h_2\left(X_{i_3},X_{i_4}\right)\right]\right|\leq C\frac{1}{\epsilon^2}\sum_{l=1}^{\infty}2^{-\eta l}<\infty.
\end{align*}
line (\ref{line29}) follows with the Borel-Cantelli lemma. To prove (\ref{line30}), we first have to find a bound for the second moments, using a well-known chaining technique:
\begin{multline*}
\max_{2^{l-1}\leq n<2^l}\left|b_nQ_n^\star-b_{2^{l-1}}Q_{2^{l-1}}^\star\right|\\
\leq \sum_{d=1}^{l} \max_{i=1,\ldots,2^{l-d}}\left|b_{2^{l-1}+i2^{d-1}}Q_{2^{l-1}+i2^{d-1}}^\star-b_{2^{l-1}+(i-1)2^{d-1}}Q_{2^{l-1}+(i-1)2^{d-1}}^\star\right|.
\end{multline*}
As for any random variables $Y_1,\ldots,Y_n$: $E\left(\max_{i=1,\ldots,n}\left|Y_i\right|\right)^2\leq \sum_{i=1}^n EY_i^2$, it follows that

\begin{align*}
&EE^\star\left[\left(\max_{2^{l-1}\leq n<2^l}\left|b_nQ_n^\star-b_{2^{l-1}}Q_{2^{l-1}}^\star\right|\right)^2\right]\\
\leq& l\sum_{d=1}^{l}\sum_{i=1}^{2^{l-d}}EE^\star\left[\left(b_{2^{l-1}+i2^{d-1}}Q_{2^{l-1}+i2^{d-1}}^\star-b_{2^{l-1}+(i-1)2^{d-1}}Q_{2^{l-1}+(i-1)2^{d-1}}^\star\right)^2\right]\displaybreak[0]\\
\leq& l\sum_{d=1}^{l}\sum_{i=1}^{2^{l-d}}2b_{2^{l-1}+i2^{d-1}}^2EE^\star\left[\left(Q^\star_{2^{l-1}+i2^{d-1}}-Q^\star_{2^{l-1}+(i-1)2^{d-1}}\right)^2\right]\\
&+l\sum_{d=1}^{l}\sum_{i=1}^{2^{l-d}}2\left(b_{2^{l-1}+i2^{d-1}}-b_{2^{l-1}+(i-1)2^{d-1}}\right)^2EE^\star\left[Q^{\star 2}_{2^{l-1}+(i-1)2^{d-1}}\right]\displaybreak[0]\\
=&\sum_{d=1}^{l}2b_{2^{l-1}+i2^{d-1}}^2EE^\star\left[\sum_{i=1}^{2^{l-d}}\left(Q^\star_{2^{l-1}+i2^{d-1}}-Q^\star_{2^{l-1}+(i-1)2^{d-1}}\right)^{2}\right]\\
&+l\sum_{d=1}^{l}\sum_{i=1}^{2^{l-d}}2\left(b_{2^{l-1}+i2^{d-1}}+b_{2^{l-1}+(i-1)2^{d-1}}\right)\left(b_{2^{l-1}+i2^{d-1}}-b_{2^{l-1}+(i-1)2^{d-1}}\right)\\
&\quad \cdot EE^\star\left[Q^{\star 2}_{2^{l-1}+(i-1)2^{d-1}}\right]\displaybreak[0]\\
\leq& l^26b^2_{2^{l-1}}\sum_{i_{1},i_{2},i_{3},i_{4}=1}^{2^l}\left|E\left[h_{2}\left(X_{i_{1}},X_{i_{2}}\right)h_{2}\left(X_{i_{3}},X_{i_{4}}\right)\right]\right|\leq Cl^22^{-\eta l}.
\end{align*}
In the last line we used the fact that the sequence $\left(b_{n}\right)_{n\in\mathds{N}}$ is decreasing, Lemma \ref{lem13} and (\ref{line31}). It now follows for all $\epsilon>0$
\begin{equation*}
\sum_{l=1}^{\infty}P\left[\max_{2^{l-1}\leq n<2^l}\left|a_nQ_n-a_{2^{l-1}}Q_{2^{l-1}}\right|>\epsilon\right]\leq\frac{C}{\epsilon^2}\sum_{l=1}^{\infty}l^22^{-\eta l}<\infty,
\end{equation*}
the Borel-Cantelli Lemma completes the proof of line (\ref{line30}). Furthermore, we have that
\begin{equation*}
\left(E\left[Q_n^\star\right]\right)^2\leq E\left[Q_n^{\star2}\right]
\end{equation*}
and conclude that $\frac{1}{\sqrt{pk}\left(pk-1\right)}E\left[Q_n^\star\right]\rightarrow 0$ a.s. We use now the Hoeffding-decomposition
\begin{multline*}
 \sqrt{pk}\left(U^{\star}_{n}\left(h\right)-E^\star\left[U^{\star}_{n}\left(h\right)\right]\right)=\frac{2}{\sqrt{pk}}\sum_{i=1}^{pk}\left(h_{1}\left(X_{i}^{\star}\right)-E^\star\left[h_{1}\left(X_{i}^{\star}\right)\right]\right)\\
+\frac{2}{\sqrt{pk}\left(pk-1\right)}\left(\sum_{1\leq i<j\leq pk}h_{2}\left(X^{\star}_{i},X^{\star}_{j}\right)-E^\star\left[\sum_{1\leq i<j\leq pk}h_{2}\left(X^{\star}_{i},X^{\star}_{j}\right)\right]\right).
\end{multline*}
By Proposition 2.11 and Lemma 2.15 of Borovkova et al. \cite{boro}, we have that $\left(h_1(X_n)\right)_{n\in\N}$ is near epoch dependent with approximation constants $\tilde{a_l}\leq C\left(\sqrt{a_l}+a_l^{\frac{3+\delta}{8+\delta}}\right)$. So we can apply Theorem \ref{theo6} to obtain
\begin{multline*}
 \sup_{x\in\mathds{R}}\left|P^{\star}\left[\frac{2}{\sqrt{pk}}\sum_{i=1}^{pk}\left(h_{1}\left(X_{i}^{\star}\right)-E^\star h_{1}\left(X_{i}^{\star}\right)\right)\leq x\right]-P\left[\frac{2}{\sqrt{n}}\sum_{i=1}^{n}h_{1}\left(X_{i}\right)\leq x\right]\right|\\
\rightarrow0\ \text{ a.s.}
\end{multline*}
and by Theorem 2.1 of Dehling, Wendler \cite{deh2}
\begin{equation*}
\sqrt{n}U_n\left(h_2\right)\rightarrow0\ \text{ a.s.}
\end{equation*}
Since $\sqrt{pk}U_n^{\star}\left(h_2\right)\rightarrow0$, $\sqrt{pk}E^\star U_n^{\star}\left(h_2\right)\rightarrow0$ a.s. have been already proved, (\ref{line17}) follows with the Lemma of Slutzky. To prove (\ref{line16}), first recall that by Theorem \ref{theo6}
\begin{equation*}
\Var^\star\left[\sqrt{pk}\sum_{i=1}^{kp}h_1\left(X_i^\star\right)\right]-\Var\left[\sqrt{n}\sum_{i=1}^{n}h_1\left(X_i\right)\right]\rightarrow0\ \text{a.s.},
\end{equation*}
and by Lemma \ref{lem12} $\Var \sqrt{n}U_n\left(h_2\right)\rightarrow0$. Similar to the proof of (\ref{line33}), one can show that $\Var^\star \sqrt{pk}U_n^\star\left(h_2\right)\ \rightarrow0$ a.s., so (\ref{line16}) follows, which completes the proof.
\end{proof}

\section*{Acknowledgment}

We are very grateful to Herold Dehling, who encouraged us to study this topic and discussed it with us many times. Olimjon Sharipov was supported by DFG (German Research Foundation). Martin Wendler was supported by the Studienstiftung des deutschen Volkes.


\begin{thebibliography}{xxxxx}
	\bibitem[1992]{arco}Arcones, M.A. and Gin\'{e}, E. (1992), 'On the bootstrap for U and V statistics', {\slshape Ann. Statist.} {\bfseries 20} 655-674.
	\bibitem[1981]{bick}Bickel, P.J. and Freedman, D.A. (1981), 'Some asymptotic theory for the bootstrap', {\slshape Ann. Statist.} {\bfseries 9} 1196-1217.
	\bibitem[1968]{bill}Billingsley, P. (1968), {\slshape Convergence of Probability Measures}, Wiley, New York.
	\bibitem[2001]{boro}Borovkova, S., Burton, R. and Dehling, H. (2001), 'Limit theorems for functionals of mixing processes with applications to $U$-statistics and dimension estimation', {\slshape Trans. Amer. Math. Soc.} {\bfseries 353} 4261--4318.
	\bibitem[2007]{bra2}Bradley, R.C. (2007), {\slshape Introduction to strong mixing conditions}, Vol. 1-3, Kendrick Press, Heber City, Utah.
	\bibitem[1986]{carl}Carlstein, E. (1986), 'The use of subseries values for estimating the variance of a general statistic from stationary sequence', {\slshape Ann. Statist.} {\bfseries 14} 1171-1179.
	\bibitem[1994]{deh3}Dehling, H. and Mikosch, T. (1994), 'Random quadratic forms and the bootstrap for U-statistics', {\slshape J. Multivariate Anal.} {\bfseries 51} 392-413.
	\bibitem[2010a]{dehl}Dehling, H. and Wendler, M. (2010), 'Central limit theorem and the bootstrap for $U$-statistics of strongly mixing data', {\slshape J. Multivariate Ana.} {\bfseries 101} 126-137.
	\bibitem[2010b]{deh2}Dehling, H. and Wendler, M. (2010), 'Law of the iterated logarithm for $U$-statistics of weakly dependent observations', in: Berkes, Bradley, Dehling, Peligrad, Tichy (Eds): Dependence in Probability, Analysis and Number Theory, Kendrick Press, Heber City.
	\bibitem[1986]{denk}Denker, M. and Keller, G. (1986), 'Rigorous statistical procedures for data from dynamical systems', {\slshape J. Stat. Phys.} {\bfseries 44} 67-93.
	\bibitem[1994]{douk}Doukhan, P. (1994), {\slshape Mixing}, Springer, New York.
	\bibitem[2002]{gonc}Gon\c{c}alves, S., White, H. (2002), 'The bootstrap of the mean fr dependent hetereogeneous arrays', {\slshape Econometric Theory} {\bfseries 18} 1367-1384.
	\bibitem[1991]{hans}Hansen, B.E. (1991), 'GARCH(1,1) processes are near epoch dependent', {\slshape Econom. Lett.} {\bfseries 36} (1991) 181-186.
	\bibitem[1948]{hoef}Hoeffding, W. (1948), 'A class of statistics with asymptotically normal distribution', {\slshape Ann. Math. Stat.} {\bfseries 19} (1948) 293-325.
	\bibitem[1982]{hofb}Hofbauer, F. and Keller, G. (1982), 'Ergodic properties of invariant measures for piecewise monotonic transformations', {\slshape Math. Z.} {\bfseries 180} 119-142.
	\bibitem[1962]{ibra}Ibragimov, I.A. (1962), 'Some limit theorems for stationary processes', {\slshape Theory Prob. Appl.} {\bfseries 7} 349-382.
	\bibitem[2003]{lahi}Lahiri, S.N. (2003), {\slshape Resampling methods for dependent data}, Springer, New York.
	\bibitem[2009]{leuc}Leucht, A. and Neumann, M.H. (2009), 'Consistency of general bootstrap methods for degenerate $U$-type and $V$-type statistics', {\slshape J. Multivariate Anal.} {\bfseries 100} 1622-1633.
	\bibitem[1998]{peli}Peligrad, M. (1998), 'On the blockwise bootstrap for empirical processes for stationary sequences', {\slshape Ann. Probab.} {\bfseries 2} 877-901.
	\bibitem[1996]{radu}Radulovic, D. (1996), 'The bootstrap of the mean for strong mixing sequences under minimal conditions', {\slshape Statist. Probab. Lett.} {\bfseries 28} 65-72.
	\bibitem[1995]{rio}Rio, E. (1995). A maximal inequality and dependent Marcinkiewicz-Zygmund strong laws. {\slshape Ann. Probab.} {\bfseries 23} 918-937.
	\bibitem[1981]{rugh}Rugh (1981), W.J., {\slshape Nonlinear system theory, the Volterra/Wiener approach}. The Johns Hopkins
University Press.
	\bibitem[1970]{serf}Serfling, R. J. (1970), 'Moment inequalities for the maximum cumulative sum', {\slshape Ann. Math. Statist.} {\bfseries 41} 1227-1234.
	\bibitem[1993]{sha2}Shao, Q.M. (1993). Complete convergence for $\alpha$-mixing sequences. {\slshape Statist. Probab. letters} {\bfseries 16} 279-287.
	\bibitem[1993]{shao}Shao, Q.M. and Yu, H. (1993), 'Bootstrapping the sample means for stationary mixing sequences', {\slshape Stochastic Process. Appl.} {\bfseries 48} 175-190.
	\bibitem[1980]{yoko}Yokoyama, R. (1980), 'Moment bounds for stationary mixing sequences', {\slshape Z. Wahrsch. verw. Gebiete} {\bfseries 52} 45-57.
	\bibitem[1976]{yosh}Yoshihara, K. (1976), 'Limiting behavior of $U$-statistics for stationary, absolutely regular processes', {\slshape Z. Wahrsch. verw. Gebiete} {\bfseries 35} 237-252.
\end{thebibliography}
\end{document}